\documentclass[11pt]{amsart}

\usepackage{amsmath,amssymb,amsthm}
\usepackage{amsmath,amssymb,amsthm,amscd}
\usepackage{enumerate}
\usepackage[frame,cmtip,arrow,matrix,line,graph,curve]{xy}
\usepackage{graphpap, color}
\usepackage[mathscr]{eucal}
\usepackage{color}
\numberwithin{equation}{section}
\newtheorem{prop}{Proposition}
\newtheorem{theo}[prop]{Theorem}
\newtheorem{lemm}[prop]{Lemma}
\newtheorem{coro}[prop]{Corollary}

\newtheorem{rema}[prop]{Remark}
\newtheorem{defi}[prop]{Definition}

\theoremstyle{definition}
\newtheorem*{pfthmsl}{Proof of Theorem~\ref{rigidity}}
\newtheorem*{ack}{Acknowledgment}
\theoremstyle{remark}

\newcommand{\p}{\partial}

\begin{document}
\title{The Weyl  problem in warped product space}

\author{Chunhe Li}
\address{School of Mathematical Sciences  \\ University of Electronic Science and Technology of China \\ Chengdu, China} \email{chli@fudan.edu.cn}
\author{Zhizhang Wang}
\address{School of Mathematical Sciences\\ Fudan University \\ Shanghai, China}
\email{zzwang@fudan.edu.cn}
\thanks{Research of the first author is supported partially by a NSFC Grant No.11571063, and the last author is supported  by a NSFC Grant No.11301087}
\begin{abstract}
In this paper, we discuss the Weyl problem in warped product space. We obtain the openness, non rigidity and some applications. These results together with the a priori estimates obtained by Lu imply some existence results. Meanwhile we reprove the infinitesimal rigidity in the space forms.
\end{abstract}

\maketitle
\section{introduction}
 The isometric embedding problem is one of the fundamental problems in differential geometry. Among them, Weyl problem is very important. It  is a milestone in the development of nonlinear elliptic partial differential equation, especially, Monge-Amp\'ere type.  In 1916, Weyl  proposed the following problem. Does every smooth metric on two dimensional sphere with positive Gauss curvature admit  a smooth isometric embedding in three dimensional Euclidean space.  Weyl \cite{W} suggested the method of continuity to solve his problem. He also  gave the openness part for the analytic case and  established an estimate on the mean curvature of the embedded strictly convex surfaces which, in fact, is  $C^2$ a priori estimate for the embedding. The analytic case was fully solved by Lewy \cite{L}. In 1953, Nirenberg, in his celebrated paper \cite{N1}, generalized  Weyl problem to the smooth case and exhibited a beautiful proof. Alexandrov and Pogorelov \cite{P1} used different approach to solve the problem independently. Pogorelov \cite{P2} also generalized Nirenberg's theorem to hyperbolic space $\mathbb{H}^{3}_{-\kappa}$  which is constant negative sectional curvature $-\kappa.$  In 1990s Weyl's estimate was generalized to degenerate case with nonnegative Gauss curvature by Guan-Li \cite{GL2}, Hong-Zuily \cite{HZ} and partially by J.A. Iaia \cite{I}. Recently, Chang-Xiao \cite{CX} and Lin-Wang \cite{LW} also consider the degenerate case for Pogorelov's theorem in hyperbolic space.
 More material about isometric embedding in Euclidean space can be found in \cite{HH}.

In general relativity, it is a difficult problem to define suitable quasi-local mass.  In 1993, using mean curvature, Brown and York \cite{BY} proposed the following definition:
Let $(\Omega, g)$ be some compact Riemannian $3$ manifold, and suppose its boundary $\p \Omega$ has positive Gauss curvature, they define the following quantity
\begin{eqnarray}
m_{BY}(\p \Omega)=\frac{1}{8\pi}\int_{\p\Omega}(H_0-H)d\sigma,
\end{eqnarray}
 where $H$ and $H_0$ are mean curvature of $\p\Omega$ in original Reimannian manifold and Euclidean space if it can be isometrically embedded in $\mathbb{R}^3$.
 By Nirenberg's isometric embedding theorem, we know that there exists such isometric embedding. In 2002, an important work by Shi and Tam \cite{ST1} has shown that the Brown-York mass is nonnegative.

In \cite{LY1,LY2}, Liu-Yau introduced Liu-Yau quasi-local mass
\begin{eqnarray}
m_{LY}(\p \Omega)=\frac{1}{8\pi}\int_{\p\Omega}(H_0-|H|)d\sigma,
\end{eqnarray}
where $|H|$ is the Lorentzian norm of the mean curvature vector. They also proved the nonnegativity of their mass.
In \cite{WY1,WY2,WY3}, Wang and Yau define some new quasilocal mass generalizing Brown-York mass.  They have used the Pogorelov's work on the isometric embedding of sphere into hyperbolic space. The nonnegativity of Wang-Yau mass is also obtained in \cite{WY1, ST2}.

The definitions of the Brown-York mass and the Wang-Yau mass suggest us that isometric embedding of the sphere into some model space always plays important role for quasi local mass problem. Hence, it seems that it is valuable to generalize the Weyl problem to some other $3$ dimensional ambient space which is not a  space form.
It may be helpful for the further discussion of quasi local mass.

As we know, the Weyl problem has not been studied in general Reimannian manifold. In the present paper, we try to study the Weyl problem in $3$ dimensional warped product space.

The warped product spaces is $\mathbb{R}^n$ with some nontrivial but rotation symmetric metrics. In fact, we can write the metric in polar coordinates
\begin{eqnarray}\label{metric}
 ds^2=\frac{1}{f^2(r)}dr^2+r^2dS^2_{n-1},
 \end{eqnarray}
where $r$ is the radius parameter and $dS_{n-1}^2$ is the standard metric of $n-1$ sphere. $f$ depends only on $r$ which is called warped function. It is well known that if we choose the warped function to be
\begin{eqnarray}\label{spaceform}
f(r)=\sqrt{1+\kappa r^2},
\end{eqnarray}
then the corresponding warped product space becomes space form with the sectional curvature $\kappa$.  If the warped function is in the form
\begin{eqnarray}\label{AdS}
f(r)=\sqrt{1-\frac{m}{r}+\kappa r^2},
\end{eqnarray}
these spaces are called Anti-de Sitter-Schwarzschild (or AdS for short) space. They are special interesting examples to
generalize quasi local mass.

Now we give the outline and main results in this paper.
The continuity method was used in the proof of Weyl problem.
As in \cite{LW}, using normalized  Ricci flow we can find the homotopy path to connect standard metric and the given metric $g$ on $\mathbb{S}^2.$
Hence we need to deal with closeness, openness, rigidity and convexity. Recently, the closeness has been obtained in \cite{GL1}. In the first part of this paper, we establish the openness for strictly convex surface in any $3$-dimensional warped product spaces.

\begin{theo}\label{openness}
Let $g$ be some smooth metric on sphere $\mathbb{S}^2$. Suppose $g$ can be isometrically embedded into the three dimensional warped product space as a closed strictly convex surface. Then for any $\alpha\in (0,1)$, there exists a positive $\epsilon$, depending only on $g$ and $\alpha$, such that, for any smooth metric
$\tilde{g}$ on $\mathbb{S}^2$ satisfying
$$\|g-\tilde{g}\|_{C^{2,\alpha}(\mathbb{S}^2)}<\epsilon,$$
 $\tilde{g}$ also can be isometrically embedded into the warped product space as another closed strictly convex surface.
\end{theo}
This theorem is proved by solving the linearized problem of isometric system. In fact, we have proved that the linearized
system always can be solved without so called infinitesimal rigidity which is different from the Nirenberg's original proof. We also note the recent work of A.J.C. Pacheco and P. Miao \cite{PM}. They have proved the openness near the standard sphere metric in Schwarzschild manifolds with small mass.  Their paper reminds us that Theorem \ref{openness} implies that openness holds for any strictly convex surface in conformal flat $3$ manifolds, since conformal flat manifolds are also warped product space.

Combining the openness with the closeness theorem, we can obtain some existence results. One of them is the following.
\begin{theo}\label{theorem2}
Suppose the three dimensional warped product space $M$ do not have singularity and the warped function $f$ satisfies the assumption (a) and (b). Then for any metric  $g$ on  $\mathbb{S}^2,$ if its Gauss curvature $K>K_0$ for some constant $K_0,$ there is some isometric embedding of $(\mathbb{S}^2,g)$  $\vec{r}$  in $M$, such that $\vec{r}$ is some strictly convex surface.
\end{theo}
 Here the assumption (a), (b) and the constant $K_0$ is defined by \eqref{assumption} in section $5$. The above existence theorem is proposed and proved by Guan-Lu \cite{GL1} and Lu \cite{Lu}.

In the second part, we discuss the infinitesimal rigidity for space form. It is well known that the infinitesimal rigidity for closed convex surface is obtained by Cohn-Vossen \cite{CV2}. Then it is simplified by Blaschke in \cite{B1} using Minkowski identities.  The infinitesimal rigidity in hyperbolic space is also discussed by Lin-Wang \cite{LW}. In this paper, we give an alternative proof.
\begin{theo}\label{rigidity}
Suppose $M$ is a convex compact surface in  three dimensional  space  form, then it's infinitesimally rigid. Namely, the solutions to the linearized problem come from the Lie algebra of the isometry group of the ambient space.
\end{theo}
In the Euclidean space, we give two proofs only using maximum principle. The idea comes from \cite{GWZ} and \cite{HNY}. For hyperbolic and spherical space cases, we use Beltrami map to extend the infinitesimal rigidity in Euclidean space to space forms.
It is different from the original proof. The Beltrami map also can be used to extend the openness from Euclidean space to space form.

The third part discusses the rigidity for isometric problem in warped product space.
The global rigidity is also obtained by Cohn-Vossen for convex surfaces \cite{CV1}. The space form cases are still valid \cite{D, GS}.
Different from the space form case, in general warped product spaces except space form  the rigidity is always not true.
For any strictly convex surface, since the isometry group will  become smaller, we always can prove the non infinitesimal rigidity. Combing this fact with the idea in the proof of the openness,
we can obtain global non rigidity for any strictly convex surface. But in fact, we can have some a little weak non rigidity for geodesic sphere in any dimension.
Explicitly, we have proved that,
\begin{theo}\label{NR}
In $n$ dimensional warped product space, the $r$ radius geodesic sphere is not rigid if the function $$\frac{ff'}{r}+\frac{1-f^2}{r^2}$$ is non zero at $r$. It means that there exists some smooth perturbation convex body isometric to the $r$ sphere but their second fundamental forms are different. If every $r$ sphere is rigid, then the ambient space should be space form.
\end{theo}
Since the scalar curvature is only determined by metric, the Alexandrov's type theorem for constant scalar curvature in general warped product space is also not right.

The reason why these examples exist is that the kernel of the linearized problem is six dimensional, but for the ambient space, the isometry group is of less than six dimension in general. See section $7$ for more explanation. It suggested us that we need to restrict its geometric center to remove the "moving" directions. Hence we introduce the following condition: \\

\noindent {\bf Condition:}
Suppose $\Sigma$ is a hypersurface and $\vec{r}$ is its position vector. We require
\begin{eqnarray}\label{Cond}
\int_{\Sigma}\vec{r}d\sigma=0,
\end{eqnarray}
where $d\sigma$ is the volume element of $\Sigma$. 

Now we can recover the rigidity for sphere under the above condition.
\begin{theo}\label{RC}
Suppose manifold $M$ is a $n-1$ dimensional topological sphere and $g$ is some Einstein metric with positive constant scalar curvature on $M$. If $(M,g)$ can be isometrically embedded into $n$ dimensional warped product space and the embedded hypersurface $\Sigma$ is  $\sigma_2$-convex, star shaped and satisfying condition \eqref{Cond}, then $\Sigma$ is a slice sphere.
\end{theo}

In the last part, we discuss some application for our theorem. We will prove some inequality similar to Shi-Tam in asymptotic flat manifolds, namely, the warped function defined by \eqref{AdS} with $\kappa=0$.
\begin{theo}\label{ST}
Suppose $\Omega$ is a mean convex and simply connected domain in some $3$ dimensional Riemannian manifold. $\Sigma_0$ is its boundary. We also assume that the induced metric on $\Sigma_0$ is in the neighborhood of the canonical sphere metric with radius $r>m$. Then $\Sigma_0$ can be isometrically embedded in Schwarzschild manifold with mass $m$ and we have the following inequality
\begin{eqnarray}\label{LM}
\int_{\Sigma_0}(H_0-H)f(\vec{r})d\sigma_0+\frac{m}{2}\geq 0.
\end{eqnarray}
Here $H_0,H$ are the mean curvature of $\Sigma_0$ in Schwarzschild manifold and original Riemannian manifold. $\vec{r}$ is the position vector of the isometric embedding. $f$ is the warped function.
The equality will hold if the $\Sigma_0$ is an asymptotic Euclidean ball.
\end{theo}
We also calculate the example constructed in Theorem \ref{NR}, and it seems that \eqref{LM} holds even if $m$ is droped in the example.

The paper is divided into ten sections. Section $2$ gives a brief review some basic formulas in warped product space. Section $3$ solves the linearized problem. Section $4$ proves the openness theorem. Section $5$ discusses some existence results. In section $6$, we revisit the infinitesimal problem in space form. Section $7$ gives the non infinitesimal rigidity and non rigidity in general warped product space. In section $8$, the counterexample of non rigidity spheres in high dimensional spaces is constructed. Section $9$ gives the proof of Theorem \ref{RC}. In section $10$, we give a Shi-Tam type inequality. In the last section, more detailed calculation on the counterexamples is presented.

\section{The connection and curvature in warped product spaces}
In $n$ dimensional  space, by \eqref{metric}, we write the warped product metric $$ds^2=\frac{1}{f^2(r)}dr^2+r^2(\sum_{i=1}^{n-1}\cos^2\theta_{i-1}\cdots\cos^2\theta_1d\theta^2_i),$$
where$(r,\theta_1,\cdots,\theta_{n-1})$ is the polar coordinate. In this paper, we always use $\cdot$  to present the metric in the ambient space.  We use $D$ as the Levi-Civita connection with respect to the metric $ds^2$. Suppose $X,Y$ are two vector fields in warped product space. We define the curvature tensor 
$$\bar{R}(X,Y)=D_XD_Y-D_YD_X-D_{[X,Y]}.$$

Then we have the following list of connections
\begin{eqnarray}\label{2.1}
\\
&&D_{\frac{\p}{\p r}}\frac{\p}{\p r}=-\frac{f'}{f}\frac{\p}{\p r}; D_{\frac{\p}{\p \theta_i}}\frac{\p}{\p r}=\frac{1}{r}\frac{\p}{\p \theta_i} \text{ for } i=1,\cdots,n-1\nonumber; \\
&&D_{\frac{\p}{\p \theta_j}}\frac{\p}{\p \theta_i}=D_{\frac{\p}{\p \theta_i}}\frac{\p}{\p \theta_j}=-\frac{\sin \theta_j}{\cos \theta_j}\frac{\p}{\p \theta_i} \text{ for } i,j=1,\cdots, n-1, \text{ and } j\leq i-1;\nonumber\\
&& D_{\frac{\p}{\p\theta_i}}\frac{\p}{\p \theta_i}=-rf^2\cos^2\theta_{i-1}\cdots\cos^2\theta_1\frac{\p}{\p r}+\sum_{j=1}^{i-1}\cos^2\theta_{i-1}\cdots\cos^2\theta_{j+1}\cos\theta_j\sin\theta_j \frac{\p}{\p \theta_j}\nonumber\\
&& \text{ for } i=1,\cdots,n-1.\nonumber
\end{eqnarray}
The rest ones are zero. Define the orthogonal frame, for $i=1,\cdots, n-1,$
\begin{eqnarray}\label{2.2}
E_1=f\frac{\partial}{\partial r},
E_{i+1}=\frac{1}{r\cos\theta_{i-1}\cdots\cos\theta_1}\frac{\partial}{\partial \theta_i}.
\end{eqnarray}
Let
$$\bar{R}_{ijkl}=\bar{R}(E_i,E_j)E_k\cdot E_l.$$
A direct calculation shows
 \begin{eqnarray}\label{curv}
&&\bar{R}_{1m1m}=\frac{ff'}{r}, \text{ for } m\neq 1;
\bar{R}_{ijij}=\frac{f^2-1}{r^2}, \text{ and }  i,j \neq 1, i\neq j; \\
&& \bar{R}_{ijkl}=0 \text{ for } i,j,k,l \text{ take three different indices.} \nonumber
\end{eqnarray}
Correspondingly, the Ricci curvatures are
 \begin{eqnarray}\label{2.4}
&&\bar{R}_{11}=\frac{(n-1)ff'}{r}; \bar{R}_{kk}=\frac{ff'}{r}+(n-2)\frac{f^2-1}{r^2}, \text{ for } k\neq 1;\\
&&\bar{R}_{ij}=0, \text{ for other cases, }\nonumber
\end{eqnarray}
and the scalar curvature is
\begin{eqnarray}\label{2.5}
\bar{R}=\bar{R}_{11}+\sum_{k=2}^n\bar{R}_{kk}=\frac{2(n-1)ff'}{r}+(n-1)(n-2)\frac{f^2-1}{r^2}.
\end{eqnarray}

Suppose $M$ is some hypersurface in $n$ dimensional warped product space and $\kappa_i$ is the $i$-th principal curvature, by the Gauss equation we have
\begin{eqnarray}\label{2.6}
\sigma_2(\kappa_1,\cdots,\kappa_{n-1})=-\frac{R}{2}+\sum_{i<j}\bar{R}_{ijij},
\end{eqnarray}
where  $e_i$ is the orthonormal frame on $M$ and $R$ is the scalar curvature of the hypersurface $M$.
Denote $\nu$ the unit outer normal. We let
$$e_n=\nu=\sum_i\nu^iE_i,$$ where $\nu^i$ is the component in the direction $E_i$ of $\nu$. Hence we have
\begin{eqnarray}
\sum_{i<j}\bar{R}_{ijij}=\frac{\bar{R}}{2}-\bar{Ric}(\nu,\nu),  \bar{Ric}(\nu,\nu)=(\nu^i)^2\bar{R}_{ii}.\nonumber
\end{eqnarray}
By \eqref{2.4}, we have
\begin{eqnarray}\label{2.7}
\sum_{i<j}\bar{R}_{ijij}&=&(n-1)\frac{ff'}{r}+\frac{(n-1)(n-2)}{2}\frac{f^2-1}{r^2}\\
&&-(\nu^1)^2(n-1)\frac{ff'}{r}-\frac{1-(\nu^1)^2}{r^2}[rff'+(n-2)(f^2-1)]\nonumber\\
&=&(n-2)[\frac{ff'}{r}+\frac{n-3}{2}\frac{f^2-1}{r^2}-\frac{(\nu^1)^2}{r^2}(rff'+1-f^2)].\nonumber
\end{eqnarray}
The conformal Killing vector in warped product space is defined by
\begin{eqnarray}\label{CK}
X&=&rf(r)\frac{\p}{\p r}.
\end{eqnarray}
We also define the distance function and support function $$\rho=\frac{1}{2}X\cdot X=\frac{r^2}{2}, \text{ and }  \varphi=X\cdot \nu.$$
Hence taking derivative with respect to $e_i$, we have
\begin{eqnarray}\label{2.9}
\rho_i&=&fX\cdot e_i,\\
\rho_{i,j}&=&\frac{f_\rho}{f}\rho_i\rho_j+f^2g_{ij}-h_{ij}f\varphi,\nonumber
\end{eqnarray}
where $h_{ij}$ is the second fundamental form of hypersurface $M$.

We give  the $r$ geodesic sphere in warped product space as a last example. We have
\begin{eqnarray}\label{sphere}
\nu=f\frac{\p}{\p r},\ \  \varphi=r,\ \ \rho=\frac{r^2}{2}, \ \  h_{ij}=\frac{f(r)}{r}g_{ij}.
\end{eqnarray}

\section{The linearized problem  }
In this section, we try to study the linearized problem in three dimensional warped product space. In three dimensional case, the metric of the ambient space can be rewritten as
$$ds^2=\frac{1}{f^2(r)}dr^2+r^2(d\theta^2+\cos^2\theta d\varphi^2).$$
We will use $\cdot$ to denote this metric. The isometric embedding problem is to find some position vector fields $\vec{r}$ to satisfy the following system
$$d\vec{r}\cdot d\vec{r}=g,$$ where $(\mathbb{S}^2,g)$ is some two dimensional sphere with given metric $g$. The linearized problem is the following system, for any two symmetric tensor $q$,
$$d\vec{r}\cdot D\tau=q,$$ where $\tau$ is the variation of $\vec{r}$ and $D$ is the Levi-Civita connection with respect to $ds^2.$
Now, let's check the above system.
In fact, for a family of surfaces $\vec{r}_t:\mathbb{S}^2\mapsto \mathbb{R}^3$ we have $$d\vec{r}_t\cdot d\vec{r}_t=g_t.$$ Taking derivative with respect to $t$, we have
$$\left.\left(\frac{\p}{\p t}d\vec{r}_t\cdot d\vec{r}_t\right)\right|_{t=0}=2d\vec{r}\cdot (D_{\frac{\p}{\p t}}d\vec{r}_t)|_{t=0}=2q.$$
Hence, it suffices to check that
$$(D_{\frac{\p}{\p t}}(\vec{r}_t)_i)|_{t=0}=D_i\tau.$$
The above equation has been shown at first in \cite{LW}. Here we give a rough review.
Suppose the coordinate of $\mathbb{R}^3$ is $\{ z^1, z^2, z^3\},$ hence
$(\vec{r}_t)_i=(\vec{r}_t)_i^m\frac{\p}{\p z^m}$. Then we have
\begin{eqnarray}
D_{\frac{\p}{\p t}}(\vec{r}_t)_i=\frac{\p (\vec{r}_i^m)}{\p t}\frac{\p}{\p z^m}+(\vec{r}_t)^m_iD_{\frac{\p}{\p t}}\frac{\p}{\p z^m}.\nonumber
\end{eqnarray}
At $t=0,$ we also write $\tau=\tau^l\frac{\p}{\p z^l}$ since $\frac{\p}{\p t}=\tau.$ By Levi-Civita property, we have
\begin{eqnarray}
(D_{\frac{\p}{\p t}}(\vec{r}_t)_i)|_{t=0}&=&\tau_i^m\frac{\p}{\p z^m}+\vec{r}_i^mD_{\tau}\frac{\p}{\p z^m}\nonumber\\
&=&\tau^m_i\frac{\p}{\p z^m}+\tau^l\vec{r}^m_iD_{\frac{\p}{\p z^l}}\frac{\p}{\p z^m}\nonumber\\
&=&\frac{\p \tau^m}{\p x^i}\frac{\p}{\p z^m}+\tau^lD_{\frac{\p}{\p x^i}}\frac{\p}{\p z^l}\nonumber\\
&=&D_i\tau.\nonumber
\end{eqnarray}

 Let $$\omega=\tau\cdot d\vec{r}=u_1dx^1+u_2dx^2,$$ where $u_i=\tau\cdot \vec{r}_i, \phi=\tau\cdot \nu$ and $\nu$ is the unit normal of the surface $\vec{r}$. Denote the connection on $S^2$ with respect to the metric $g$ by $\nabla,$ then we have
\begin{eqnarray}\label{3.1}
\nabla \omega&=&du_i\otimes dx^i+u_i\nabla dx^i\\
&=&\vec{r}_i\cdot D\tau\otimes dx^i+\tau\cdot D\vec{r}_i\otimes dx^i+u_i\nabla dx^i\nonumber\\
&=&\vec{r}_i\cdot D_j\tau dx^i\otimes dx^j+\tau\cdot \nabla r_i\otimes dx^i+\tau\cdot D^{\perp}_{\vec{r}_j}\vec{r}_i dx^j\otimes dx^i+u_i\nabla dx^i.\nonumber
\end{eqnarray}
By the definition of the second fundamental form, we have
$$D^{\perp}_{\vec{r}_j}\vec{r}_i =h_{ij}\nu.$$ By the definition of the connection, we have
$$\nabla \vec{r}_i=\Gamma_{ij}^k\vec{r}_k\otimes dx^j.$$
Since $\Gamma^k_{ij}=\Gamma^k_{ji}$, we have
\begin{eqnarray}\label{3.2}
\nabla \omega&=&\vec{r}_i\cdot D_j \tau dx^i\otimes dx^j+\phi h_{ij}dx^i\otimes dx^j\\&&+\Gamma_{ij}^ku_kdx^j\otimes dx^i-u_i\Gamma^i_{ml}dx^m\otimes dx^l\nonumber\\
&=&(\vec{r}_i\cdot D_j \tau+\phi h_{ij})dx^i\otimes dx^j.\nonumber
\end{eqnarray}
The above equation is equivalent to the following system
\begin{eqnarray}\label{3.3}
\left \{ \begin{matrix} u_{1,1}&=&q_{11}+\phi h_{11} \\  u_{1,2}+u_{2,1}&=&2(q_{12}+\phi h_{12})\\  u_{2,2}&=& q_{22}+\phi h_{22}\end{matrix}\right..
\end{eqnarray}
Denote the symmetrization operator by $\mathrm{Sym}$, namely  $$\mathrm{Sym}:T^*\mathbb{S}^2\otimes T^* \mathbb{S}^2\mapsto \mathrm{Sym}(T^*\mathbb{S}^2\otimes T^* \mathbb{S}^2).$$ Then we define an operator
\begin{eqnarray}
L: T^*\mathbb{S}^2&\mapsto &\mathrm{Sym}(T^*\mathbb{S}^2\otimes T^* \mathbb{S}^2)\nonumber\\ 
L(\omega)&=&\mathrm{Sym}(\nabla \omega).\nonumber
\end{eqnarray}
Hence \eqref{3.3} becomes   $$L(\omega)=q+\phi h.$$
Taking trace with respect to $h$, we have
$$\mathrm{tr}_h(\nabla \omega)=h^{ij}u_{i,j}=\mathrm{tr}_h(q)+2\phi.$$ Then  we have $$\phi= \frac{\mathrm{tr}_h(\nabla\omega-q)}{2}.$$
For convenience, define a new operator
\begin{eqnarray}\label{L_h}
L_h(\omega)=L(\omega)-\frac{\mathrm{tr}_h(\nabla \omega)}{2}h=q-\frac{\mathrm{tr}_h(q)}{2}h,
\end{eqnarray}
which is a map from cotangent bundle to symmetric $2$-tensor product of cotangent bundle. We claim that  $L_h$ is a strong elliptic operator. For $\xi=(\xi_1,\xi_2)$, $|\xi|=1$, if the principal symbol of the operator $L_h$ is zero, namely
$$\sigma(L_h)=0,$$ which implies
\begin{eqnarray}\label{3.4}
\left \{ \begin{matrix} u_{1}\xi_1-\eta h_{11}&=&0 \\  u_{1}\xi_2+u_{2}\xi_1-2\eta h_{12}&=& 0 \\  u_{2}\xi_2-\eta h_{22}&=& 0\end{matrix}\right.,
\end{eqnarray}
where $2\eta=h^{ij}u_i\xi_j$. If $\xi_1,\xi_2\neq 0,$ then $u_1=\eta h_{11}/\xi_1$ and $u_2=\eta h_{22}/\xi_2$. Hence  we have
$$0=\eta(\frac{\xi_2h_{11}}{\xi_1}+\frac{\xi_1 h_{22}}{\xi_2}-2h_{12})=\frac{\eta}{\det h \xi_1\xi_2}(h^{ij}\xi_i\xi_j).$$ Since $h$ is positive definite, we get $\eta=0$ which implies $u_1=u_2=0$.  If $\xi_1=0,$ then $\eta=0$ and $\xi_2\neq 0$. Furthermore, by the third equality of  \eqref{3.4}, we have $u_2=0$. By the second equality of \eqref{3.4}, we also have $u_1=0$.  Hence the operator $L_h$ is a strong elliptic operator on compact manifold.

 Since the strong elliptic operator is  a Fredholm operator, we know that
$$\mathrm{coker}(L_h)=\ker(L_h^*).$$ Hence to prove the solvability of the linearized problem, it suffices to show that the kernel of  $L_h^*$ is zero. Let's compute the adjoint operator. At first define some inner product space $H$
$$H:=\{a\in C^{\infty}(Sym(T^*\mathbb{S}^2\otimes T^*\mathbb{S}^2)); \mathrm{tr}_h(a)=0\} $$  equipped with the inner product
$$(a,b)=\int_{\mathbb{S}^2}Kh^{ij}h^{mn}a_{im}b_{jn}dV_g$$
 for $a,b\in H,$ where
 $K$ is $\sigma_2$ curvature  for surface $\vec{r}$ in ambient space, namely
$$K=\frac{\det h}{\det g}.$$
 We also define the inner product in $C^{\infty}(T^*\mathbb{S}^2)$, $$(\omega,\eta)=\int_{\mathbb{S}^2}Kh^{ij}\omega_{i}\eta_{j}dV_g$$ for $\omega,\eta\in C^{\infty}(T^*\mathbb{S}^2)$.

  The adjoint operator $L_h^*$ is defined by
  $$(\omega,L^*_h(a))=(L_h(\omega), a).$$
Then we have
\begin{eqnarray}\label{3.6}
\\
(L_h(\omega),a)
&=&\int_{\mathbb{S}^2}Kh^{ij}h^{mn}u_{i,m}a_{jn}dV_g+\frac{1}{2}\int_{\mathbb{S}^2}K\mathrm{tr}_h(\nabla \omega)h^{ij}h^{mn}h_{im}a_{jn}dV_g\nonumber\\
&=&-\int_{\mathbb{S}^2}(Kh^{ij}h^{mn}a_{jn})_{,m} u_idV_g.\nonumber
\end{eqnarray}
where we have used $h^{ij}a_{ij}=0.$
Hence the dual operator is
$$(L^*_h(a))_k=-(Kh^{ij}h^{mn}a_{jn})_{,m}h_{ik}.$$  Thus $L^*_h(a)=0$ implies
$$(Kh^{ij}h^{mn}a_{jn})_{,m} =0.$$
Let $A_{ij}$ be the cofactor of $h_{ij},$ the previous equation is
$$(A_{ij}h^{mn}a_{jn})_{,m}=0.$$
For  $i=1$, the above equation can be written as
\begin{eqnarray}
0&=&(A_{11}h^{1n}a_{1n}+A_{12}h^{1n}a_{2n})_{,1}+(A_{11}h^{2n}a_{1n}+A_{12}h^{2n}a_{2n})_{,2}\nonumber\\
&=&(-A_{11}h^{2n}a_{2n}+A_{12}h^{1n}a_{2n})_{,1}+(A_{11}h^{2n}a_{1n}-A_{12}h^{1n}a_{1n})_{,2}\nonumber\\
&=&(-h_{22}h^{2n}a_{2n}-h_{12}h^{1n}a_{2n})_{,1}+(h_{22}h^{2n}a_{1n}+h_{12}h^{1n}a_{1n})_{,2}\nonumber\\
&=&-(\delta_{2n}a_{2n})_{,1}+(\delta_{2n}a_{1n})_{,2}\nonumber\\
&=&a_{21,2}-a_{22,1},\nonumber
\end{eqnarray}
where  we have used $h^{ij}a_{ij}=0$ in the second equality.
 A similar argument yields $$a_{11,2}=a_{12,1}.$$
Then we obtain, for $k=1,2$,
$$a_{k1,2}=a_{k2,1}.$$ Hence the tensor $a$ satisfy homogenous linearized Gauss-Codazzi system
\begin{eqnarray}\label{3.7}
\left\{\begin{matrix}h^{ij}a_{ij}=0\\a_{ij,k}=a_{ik,j}\end{matrix}\right..
\end{eqnarray}
In what follows, we will prove that \eqref{3.7}  has  trivial solution only which implies $\ker(L_h^*)=0$.

Let's define cross product $\times$ induced by arbitrary inner product $\cdot$. It is a standard procedure. Once we fix an orientation, for any two vector $a,b, $ the cross product is the unique bilinear map satisfying the following two properties:
\begin{eqnarray}\label{3.8}
  a\times b \cdot a=a\times b\cdot b=0 ;  |a\times b|^2=(a\cdot a)(b\cdot b)-(a\cdot b)^2.
\end{eqnarray}
It is easy to check that the following triple scalar product formula holds for any vectors $a,b,c$, $$a\times b\cdot c=a\cdot b\times c.$$

We define some $(1,1)$ tensor related to the linearized Gauss-Codazzi system. Let $$A_k=g^{ij}a_{ki}\nu\times \vec{r}_j,$$ where $\nu$ is the unit normal of surface $\vec{r}$ and in the direction of $\vec{r}_1\times \vec{r}_2$. By \eqref{3.8}, $A_k$ is a tangential vector. Define
\begin{eqnarray}\label{omg}
\Omega=A_k\otimes dx^k,
\end{eqnarray}
then $\Omega$ is a $(1,1)$ tensor on sphere $\mathbb{S}^2.$ More precisely, by \eqref{3.8} and triple scalar product formula, we have
\begin{eqnarray}
\left\{\begin{matrix}\label{3.9}
A_1&=&\frac{1}{\sqrt{\det g}}(-a_{12}\frac{\p}{\p x_1}+a_{11}\frac{\p}{\p x_2})\\
A_2&=&\frac{1}{\sqrt{\det g}}(-a_{22}\frac{\p}{\p x_1}+a_{21}\frac{\p}{\p x_2})
\end{matrix}\right..
\end{eqnarray}
\begin{lemm}\label{le7}
For any vector $E$ in ambient space satisfying
\begin{eqnarray}\label{3.10}
d\vec{r}\cdot DE=0,
\end{eqnarray}
we can define some $1$- form on $\mathbb{S}^2$ $$\omega=\Omega\cdot E=A_k\cdot E dx^k.$$ Then  $\omega$ is a closed form, and in fact it is zero.
\end{lemm}
\begin{proof}
It is obvious that $\omega$ is a one form. Then
\begin{eqnarray}
d\omega&=&\p_j(A_k\cdot E)dx^j\wedge dx^k\\
&=&(D_jA_k\cdot E+A_k\cdot D_jE)dx^j\wedge dx^k\nonumber\\
&=&[(D_1A_2-D_2A_1)\cdot E+(A_2\cdot D_1E-A_1\cdot D_2E)]dx^1\wedge dx^2\nonumber.
\end{eqnarray}
By \eqref{3.10}  we have
$$\frac{\p}{\p x_1}\cdot D_1E=0;\frac{\p}{\p x_2}\cdot D_2E=0;\frac{\p}{\p x_1}\cdot D_2E+\frac{\p}{\p x_2}\cdot D_1E=0.$$
Hence, by \eqref{3.9}, we get
$$A_2\cdot D_1E-A_1\cdot D_2E=\frac{a_{21}}{\sqrt{\det g}}\frac{\p}{\p x_2}\cdot D_1E+\frac{a_{12}}{\sqrt{\det g}}\frac{\p}{\p x_1}\cdot D_2E=0.$$
For convenience, we write
$$A_k=w_k^l\vec{r}_l,$$ where $w_k^l$ is a $(1,1)$ tensor, and the relation between $a_{ij}$ and $w_k^l$ is released in \eqref{3.9}.
We decompose the connection into two parts, namely tangential and normal parts $$D=\nabla+D^{\perp}.$$ Then  we have
$$D_iA_k=(w^l_{k,i}+\Gamma_{ik}^mw_{m}^l)\vec{r}_l+w_k^lD^{\perp}_i\vec{r}_l=(w^l_{k,i}+\Gamma_{ik}^mw_{m}^l)\vec{r}_l+w_k^lh_{il}\nu,$$
 where $w^l_{k,i}$ is the covariant derivative for $(1,1)$ tensor $w^l_k$. Since $\Gamma^m_{ik}=\Gamma^m_{ki}$,  by \eqref{3.7} and \eqref{3.9}, we have $$D_1A_2-D_2A_1=(w_{1,2}^l-w_{2,1}^l)\vec{r}_l+(h_{1l}w^l_2-h_{2l}w^l_1)\nu=0.$$
Thus  $\omega$ is a closed one form.

Since the first de Rham cohomology of the sphere is trivial, $H^1_{DR}(\mathbb{S}^2)=0,$ there exists some smooth function $f$ on sphere such that
$$\omega=df=f_kdx^k.$$ Therefore we have
\begin{eqnarray}\label{3.12}
f_k=A_k\cdot E=w^l_k\vec{r}_l\cdot E.
\end{eqnarray}
Then using the previous similar formula, we have
\begin{eqnarray}
\nabla_j\omega&=&\p_j(A_k\cdot E)dx^k+A_k\cdot E \nabla_jdx^k \\
&=&(D_jA_k\cdot E+A_k\cdot D_jE)dx^k-\Gamma^k_{jl}A_k\cdot E dx^l\nonumber\\
&=&[(w^l_{k,j}+\Gamma_{jk}^mw_{m}^l)\vec{r}_l\cdot E+w_k^lh_{jl}\nu\cdot E+w_k^l\vec{r}_l\cdot D_jE]dx^k\nonumber\\
&&-\Gamma^k_{jl}w_k^m\vec{r}_m\cdot E dx^l\nonumber\\
&=&[w^l_{k,j}\vec{r}_l\cdot E+w_k^lh_{jl}\nu\cdot E+w_k^l\vec{r}_l\cdot D_jE]dx^k\nonumber.
\end{eqnarray}
Hence we get
$$f_{k,j}=w^l_{k,j}\vec{r}_l\cdot E+w_k^lh_{jl}\nu\cdot E+w_k^l\vec{r}_l\cdot D_jE.$$
Then we obtain
\begin{eqnarray}\label{3.14}
h^{kj}f_{k,j}&=&h^{kj}w^l_{k,j}\vec{r}_l\cdot E+w^k_k\nu\cdot E+h^{kj}w_k^l\vec{r}_l\cdot D_jE\\
&=&h^{kj}w^l_{k,j}\vec{r}_l\cdot E+(\frac{-a_{12}}{\sqrt{\det g}}+\frac{a_{21}}{\sqrt{\det g}})\nu\cdot E\nonumber\\
&&+h^{1k}w_k^2\vec{r}_2\cdot D_1E+h^{2k}w_k^1\vec{r}_1\cdot D_2E\nonumber\\
&=&h^{kj}w^l_{k,j}\vec{r}_l\cdot E+(h^{1k}w_k^2-h^{2k}w_k^1)\vec{r}_2\cdot D_1E\nonumber\\
&=&h^{kj}w^l_{k,j}\vec{r}_l\cdot E+\frac{h^{ij}a_{ij}}{\sqrt{\det g}}\vec{r}_2\cdot D_1E\nonumber\\
&=&h^{kj}w^l_{k,j}\vec{r}_l\cdot E\nonumber.
\end{eqnarray}
By the Lemma 4 in \cite{GWZ}, we have
$$(w^1_1)^2+(w^1_2)^2+(w^2_1)^2+(w^2_2)^2\leq -C\det w.$$
We conclude that
$$h^{kj}f_{k,j}=h^{kj}w^l_{k,j}\frac{B^m_l f_m}{\det w},$$  where $B^m_l$ is the cofactor of $w^m_l$.
We also have  for $l=1$,
\begin{eqnarray}
h^{ij}w^1_{i,j}
&=&h^{11}w^1_{1,1}+h^{12}w^1_{1,2}+h^{21}w^1_{2,1}+h^{22}w^1_{2,2}\nonumber\\
&=&\frac{1}{\sqrt{\det g}}(-h^{11}a_{12,1}-h^{12}a_{12,2}-h^{21}a_{22,1}-h^{22}a_{22,2})\nonumber\\
&=&-\frac{1}{\sqrt{\det g}}(h^{11}a_{11,2}+h^{12}a_{12,2}+h^{21}a_{21,2}+h^{22}a_{22,2})\nonumber\\
&=&-\frac{1}{\sqrt{\det g}}h^{ij}a_{ij,2}=\frac{1}{\sqrt{\det g}}h^{ij}_{,2}a_{ij}\nonumber.
\end{eqnarray}
Similarly, we have
$$h^{ij}w^2_{i,j}=\frac{1}{\sqrt{\det g}}h^{ij}_{,1}a_{ij}.$$
The strong maximum principle implies that
$f$ is a constant function on sphere which means
$$\omega=df=0.$$
\end{proof}

We will find some special solutions to \eqref{3.10}.
\begin{lemm} \label{le8}
Suppose the coordinate for $\mathbb{R}^3$ are $\{z^1, z^2, z^3\}$. Let $$E_\alpha=\frac{\p}{\p z^{\alpha}}\times_{E} \vec{r},$$ where $\times_E$ is the cross product induced by Euclidean metric, then $E_\alpha$ satisfy \eqref{3.10} for $\alpha=1,2,3$.
And we also have $$E_{\alpha}=\frac{\p}{\p z^{\alpha}}\times  X,$$ where $X$ is the conformal Killing vector.
\end{lemm}
\begin{proof}
In view of the discussion in the procedure to derive the linear equation. We only need to construct some family of embedding $\vec{r}_t$, such that
$$d\vec{r}_t\cdot d\vec{r}_t=d\vec{r}\cdot d\vec{r}.$$
Then $E=\frac{d \vec{r}_t}{dt}|_{t=0}$ satisfies \eqref{3.10}. It is obvious that  the warped product space is rotationally symmetric. Explicitly, for any orthogonal matrix $A$, we always have $A\vec{r}\cdot A\vec{r}=\vec{r}\cdot\vec{r}$.
Let $O(3)$ be the  orthogonal group in dimension $3$. We know that the Lie algebra $o(3)$ of $O(3)$  is the collection of all $3\times 3$ anti-symmetric matrices. It is well known that $o(3)$ is $3$ dimensional and its basis can be expressed by $E_{\alpha}$. Thus for any $E_{\alpha}$, we always can find some family of orthogonal matrices  $A_t,$ which is a path generated by $E_{\alpha}$ from the identity matrix $I$ in
$O(3)$ using exponential map. Let $$\vec{r}_t=A_t\vec{r},$$ we have
$$d\vec{r}_t\cdot d\vec{r}_t=A_t d\vec{r}\cdot A_t d\vec{r}=d\vec{r}\cdot  d\vec{r},$$ which implies that $E_{\alpha}$ satisfy \eqref{3.10}.

The conformal Killing vector field is defined by \eqref{CK}.
In fact, we will prove, for any vector $a$, $$a\times_E\vec{r}=a\times X.$$
Suppose the polar coordinate is $(r,\theta, \varphi)$, and write
\begin{eqnarray}\label{vec}
a=\xi\frac{\p}{\p r}+\eta\frac{\p}{\p \theta}+\zeta\frac{\p}{\p \varphi},\text{ and } \vec{r}=r\frac{\p}{\p r}.
\end{eqnarray}
 With the same orientation in $\mathbb{R}^3$, we have
$$a\times_E \vec{r}=r(\frac{\eta}{\cos\theta}\frac{\p}{\p \varphi}-\zeta\cos\theta\frac{\p}{\p \theta}),$$ which is perpendicular to $a,\vec{r}$. Thus $a\times_E \vec{r}$ is parallel to
$a\times X$. Then we obtain $$|a\times X|^2=|a|^2|X|^2-(a\cdot X)^2=r^2f^2(\zeta^2\cos^2\theta+\eta^2)=|a\times_E \vec{r}|^2.$$
The claim is proved.
\end{proof}

The support function in warped product space is defined by $$\varphi(p)=X\cdot \nu,$$ where $\nu$ is the unit outward normal of surface $\vec{r}$ and $X$ is defined by \eqref{CK}. We also define another set
$$m_0=\{\vec{r}(p)=0;p\in \mathbb{S}^2\}.$$ Since $\vec{r}$ is regular surface, $m_0$ is a finite set by the compactness of sphere.
We have the following fact,
\begin{prop}
The set $Z=\{\varphi(p)=0; p\in\mathbb{S}^2\setminus m_0\}$ is a regular curve on sphere.
\end{prop}
\begin{proof}
We only need to check that on $Z$, $d \varphi\neq 0$. Since $X$ is a conformal Killing vector, we have
$$\p_i\varphi=D_iX\cdot \nu+X\cdot D_i n=f\vec{r}_i\cdot \nu+X\cdot D_i \nu=X\cdot D_i\nu.$$ Note that $$D_i\nu\cdot \vec{r}_j=-\nu\cdot D_i\vec{r}_j=-h_{ij}.$$ Thus we get $$D_i\nu=-g^{kl}h_{ik}\vec{r}_l.$$ Since $X\cdot \nu=0$ on $Z,$ we can assume $X=a^i\vec{r}_i$. Then we have
$$\p_i z=-a^kh_{ik}=0,$$ which implies $a^k=0$, namely $X=0$. Hence we know that $\vec{r}(p)=0$ which contradicts to $p\notin m_0.$
\end{proof}

We are in the position to prove the following Theorem.
\begin{theo}\label{LE}
Suppose $\vec{r}$ is a strictly convex surface in warped product space. For any $(0,2)$ symmetric tensor $q$, there always exists some vector field $\tau$ satisfying the following equation,
$$d\vec{r}\cdot D\tau=q.$$
\end{theo}
\begin{proof} By the previous discussion, it suffices to prove $\ker(L_h^*)=0$ which is equivalent to prove $\Omega=0$ where $\Omega$ is the $(1,1)$ tensor defined in \eqref{omg}.  By Lemma \ref{le7} and Lemma \ref{le8}, for $\alpha=1,2,3$, we have, on the whole sphere, $$\Omega\cdot E_\alpha=0,\text{ and }\Omega\cdot \nu=0.$$
Since for any three vector $a,b,c$, we always have $$(a\times b)\times c=b(a\cdot c)-a(b\cdot c).$$ Hence for $\alpha\neq \beta$ and $\alpha,\beta=1,2,3$, we have
\begin{eqnarray}
E_\alpha\times E_\beta\cdot \nu&=&(\frac{\p}{\p z^{\alpha}}\times X)\times (\frac{\p}{\p z^{\beta}}\times X)\cdot \nu\\
&=&(\frac{\p}{\p z^{\alpha}}\times\frac{\p}{\p z^{\beta}}\cdot X)(X\cdot \nu)\nonumber\\
&=&(\frac{\p}{\p z^{\alpha}}\times\frac{\p}{\p z^{\beta}}\cdot X)\varphi.\nonumber
\end{eqnarray}
For any vector $X\neq 0$, we always can find $\alpha_0\neq \beta_0$ such that $\dfrac{\p}{\p z^{\alpha_0}},\dfrac{\p}{\p z^{\beta_0}}$ and $X$ are linearly independent. Since function $\varphi$ is non zero on $\mathbb{S}^2\setminus(Z\cup m_0),$ we have $E_{\alpha_0},E_{\beta_0}$ and $\nu$ are linearly independent which implies, on $\mathbb{S}^2\setminus(Z\cup m_0),$$$\Omega=0.$$  Note that $Z\cup m_0$ is a closed set without interior points. Hence $\Omega$ is zero on the whole sphere. Thus we complete our proof.

\end{proof}
\section{The openness }
Using Theorem \ref{LE}, we can give the proof of the openness Theorem \ref{openness}  for the Weyl problem.
Assume that $\vec{r}$ is the isometric embedding of $(\mathbb{S}^2,g)$ in ambient space. Following the Nirenberg's approach, we need to find some vector field $\vec{y}$, such that $$d(\vec{r}+\vec{y})\cdot d(\vec{r}+\vec{y})=\tilde{g},$$ where $\tilde{g}$ is the perturbation of $g$. We rewrite the isometric system in local coordinate of sphere $\{x_1,x_2\}$ and ambient space $\{z^1, z^2, z^3\}$,
\begin{eqnarray}\label{4.1}
\sigma_{\alpha\beta}(\vec{r}+\vec{y})\frac{\partial(\vec{r}^{\alpha}+\vec{y}^{\alpha})}{\partial x^{i}}\frac{\partial(\vec{r}^{\beta}+\vec{y}^{\beta})}{\partial x^{j}}=\tilde{g}_{ij},
\end{eqnarray}
and
$$\sigma_{\alpha\beta}(\vec{r})\frac{\partial\vec{r}^{\alpha}}{\partial x^{i}}\frac{\partial\vec{r}^{\beta}}{\partial x^{j}}=g_{ij}.$$
By \eqref{4.1}, we have
\begin{eqnarray}
&&\tilde{g}_{ij}-g_{ij}\\
&=&\sigma_{\alpha\beta}(\vec{r})(\vec{r}^{\alpha}_{i}D_{j}\vec{y}^{\beta}+\vec{r}^{\beta}_{j}D_{i}\vec{y}^{\alpha})
-\sigma_{\alpha\beta}(\vec{r})(\vec{r}^{\alpha}_{i}\Gamma^{\beta}_{j\gamma}\vec{y}^{\gamma}+\vec{r}^{\beta}_{j}\Gamma^{\alpha}_{i\gamma}\vec{y}^{\gamma})\nonumber\\
&&+\sigma_{\alpha\beta}(\vec{r})\vec{y}^{\alpha}_{i}\vec{y}^{\beta}_{j}+(\sigma_{\alpha\beta}(\vec{r}+\vec{y})-\sigma_{\alpha\beta}(\vec{r}))
(\vec{r}^{\alpha}_{i}\vec{r}^{\beta}_{j}+\vec{r}^{\alpha}_{i}\vec{y}^{\beta}_{j}
+\vec{y}^{\alpha}_{i}\vec{r}^{\beta}_{j}+\vec{y}^{\alpha}_{i}\vec{y}^{\beta}_{j})\nonumber
\end{eqnarray}
where $D_{i}=D_{\frac{\partial}{\partial x^{i}}}, \vec{r}^{\alpha}_{i}=\frac{\partial \vec{r}^{\alpha}}{\partial x^{i}}$ and
$\Gamma^{\alpha}_{i\gamma}=\vec{r}^{\lambda}_{i}\Gamma^{\alpha}_{\lambda\gamma}$. By Taylor expansion theorem, we have
$$\sigma_{\alpha\beta}(\vec{r}+\vec{y})-\sigma_{\alpha\beta}(\vec{r})=\partial_{\gamma}\sigma_{\alpha\beta}(\vec{r})\vec{y}^{\gamma}
+\vec{y}^{\gamma}\vec{y}^{\lambda}\int_{0}^{1}(1-t)\partial^{2}_{\gamma\lambda}\sigma_{\alpha\beta}(\vec{r}+t\vec{y})dt.$$ By the definition of Christoffel symbol, we have $$\partial_{\gamma}\sigma_{\alpha\beta}(\vec{r})\vec{r}^{\alpha}_{i}\vec{r}^{\beta}_{j}
=\sigma_{\alpha\beta}(\vec{r})(\vec{r}^{\alpha}_{i}\Gamma^{\beta}_{j\gamma}+\vec{r}^{\beta}_{j}\Gamma^{\alpha}_{i\gamma}).$$
Let
\begin{eqnarray}
F_{\alpha\beta\gamma\lambda}(\vec{r},\vec{y})&=&\int_{0}^{1}(1-t)\partial^{2}_{\gamma\lambda}\sigma_{\alpha\beta}(\vec{r}+t\vec{y})dt,\\
G_{\alpha\beta\gamma}(\vec{r},\vec{y})&=&\int_{0}^{1}\partial_{\gamma}\sigma_{\alpha\beta}(\vec{r}+t\vec{y})dt.\nonumber
\end{eqnarray}
Then we rewrite $q=q_{ij}dx^idx^i$ and $q_{ij}$ is
\begin{eqnarray}\label{4.4}
q_{ij}(\vec{y})&=&\tilde{g}_{ij}-g_{ij}-\sigma_{\alpha\beta}(\vec{r})\vec{y}^{\alpha}_{i}\vec{y}^{\beta}_{j}
-F_{\alpha\beta\gamma\lambda}(\vec{r},\vec{y})\vec{r}^{\alpha}_{i}\vec{r}^{\beta}_{j}\vec{y}^{\gamma}\vec{y}^{\lambda}\\
&&-G_{\alpha\beta\gamma}(\vec{r},\vec{y})\vec{y}^{\gamma}(\vec{r}^{\alpha}_{i}\vec{y}^{\beta}_{j}
+\vec{y}^{\alpha}_{i}\vec{r}^{\beta}_{j}+\vec{y}^{\alpha}_{i}\vec{y}^{\beta}_{j})\nonumber.
\end{eqnarray}
We conclude that \eqref{4.1} is equivalent to the following inhomogeneous infinitesimal-deformation type equations
\begin{eqnarray}\label{4.5}
d\vec{r}\cdot D \vec{ y}= q(\vec{y},\nabla \vec{y}).
\end{eqnarray}
The above calculation has been exhibited in \cite{LW} at first. Note that
$$\|F_{\alpha\beta\gamma\lambda}\|_{C^{m,\alpha}(\mathbb{S}^2)}+\|G_{\alpha\beta\gamma}\|_{C^{m,\alpha}(\mathbb{S}^2)}\leq C_{m,\alpha}\|\vec{y}\|_{C^{m,\alpha}(\mathbb{S}^2)},$$ where $\|\cdot\|_{C^{m,\alpha}(\mathbb{S}^2)}$ is the H\"older norm on $2$ sphere.

Theorem \ref{LE} tells us that, for any given symmetric tensor $q$, system \eqref{4.5}
always can be solved. In what follows, we will make use of contraction mapping theorem to obtain some solution to \eqref{4.5}.
We define some map $\Phi$,
\begin{eqnarray}
\Phi: C^{2,\alpha}(\mathbb{S}^2)&\mapsto &C^{2,\alpha}(\mathbb{S}^2)\nonumber,\\
\vec{z}&\mapsto &\vec{y}\nonumber,
\end{eqnarray}
where $\vec{y}=\Phi(\vec{z})$ is the solution to \eqref{4.5} with
$q=q(\vec{z},\nabla\vec{z})$. We need a priori estimate for the solution to \eqref{4.5}.

Similar to the previous section,  an  elliptic system \eqref{3.3} can be derived from \eqref{4.5} for
$$\omega=\vec{y}\cdot d\vec{r}=u_{i}dx^{i}.$$ We need its H\"older estimate.

Takeing derivatives of system \eqref{L_h}, we obtain some second order elliptic partial differential system. Using the interior Schauder estimate of elliptic system \cite{GM}, we have
\begin{eqnarray}
\|\omega\|_{C^{2,\alpha}(\mathbb{S}^2)}\leq C_{2,\alpha}(\|q\|_{C^{1,\alpha}(\mathbb{S}^2)}+\|\omega\|_{C^{0,\alpha}(\mathbb{S}^2)}).\nonumber
\end{eqnarray}
A compactness argument implies, for $\omega$ orthogonal to the kernel of the system \eqref{3.3},
\begin{eqnarray}\label{4.6}
\|u_1\|_{C^{2,\alpha}(\mathbb{S}^2)}+\|u_2\|_{C^{2,\alpha}(\mathbb{S}^2)}\leq C_{2,\alpha}\|q\|_{C^{1,\alpha}(\mathbb{S}^2)}.
\end{eqnarray}
Here, the constant $C_{2,\alpha}$ only depends on the upper and lower bound of the principal curvatures $\kappa_1,\kappa_2$ of the strictly convex surface $\vec{r}$ and the upper bound of the derivatives of the curvature.
See also Vekua \cite{V} for another proof of the above inequality.
In what follows, we will modify Nirenberg's trick \cite{N1} to improve the regularity of the function $\phi=\vec{y}\cdot \nu$ where $\nu$ is the unit outer normal of the surface $\vec{r}$. Let
\begin{eqnarray}\label{4.7}
w=\frac{u_{1,2}-u_{2,1}}{2\sqrt{|g|}},
\end{eqnarray}
where $|g|=\det(g)$. Since $d\omega=2w\sqrt{|g|}dx^{1}\wedge dx^{2}$,  $w$ is a globally defined function
on $\mathbb{S}^2$.  Moreover, $w$ satisfies
\begin{eqnarray}\label{4.8}
w\sqrt{|g|}+q_{12} &=&u_{1,2}-h_{12}\phi,\\  q_{21}-w\sqrt{|g|} &=&u_{2,1}-h_{12}\phi.\nonumber
\end{eqnarray}
By Ricci identities, we have
\begin{eqnarray}\label{4.9}
R^{l}_{121}u_{l}&=&u_{1,21}-u_{1,12}\\
&=&(q_{12}+w\sqrt{|g|}+h_{12}\phi)_{,1}-(q_{11}+h_{11}\phi)_{,2} \nonumber\\
R^{k}_{212}u_{k}&=&u_{2,12}-u_{2,21}\nonumber\\
&=&(q_{21}-w\sqrt{|g|}+h_{21}\phi)_{,2}-(q_{22}+h_{22}\phi)_{,1}\nonumber.
\end{eqnarray}
Hence we have
\begin{eqnarray}
-w_1=
\frac{1}{\sqrt{|g|}}(q_{12,1}-q_{11,2}+(h_{21}\phi)_{,1}-(h_{11}\phi)_{,2}-R^{l}_{121}u_{l}),\nonumber
\end{eqnarray}
\begin{eqnarray}
-w_2=\frac{1}{\sqrt{|g|}}(q_{22,1}-q_{12,2}+(h_{22}\phi)_{,1}-(h_{21}\phi)_{,2}+R^{k}_{212}u_{k}).\nonumber
\end{eqnarray}
Using compatibility condition $w_{1,2}=w_{2,1}$, we obtain
\begin{eqnarray}\label{4.10}
&&\frac{1}{\sqrt{|g|}}((h_{22}\phi)_{,11}-(h_{21}\phi)_{,21}-(h_{21}\phi)_{,12}+(h_{11}\phi)_{,22})\\
&=&\frac{1}{\sqrt{|g|}}(-(R^{l}_{121}u_{l})_{,2}-(R^{k}_{212}u_{k})_{,1}+q_{12,12}-q_{11,22}-q_{22,11}+q_{12,21}).\nonumber
\end{eqnarray}
It is not difficult to see that the above equation is an elliptic equation, if $h_{ij}$ is positive definite.
\begin{rema}
In fact, the above equation \eqref{4.10} is the linearized Gauss equation. Since the tensor $q$ is the linearization of the metric tensor, the curvature type tensor of $q$ appears in the right hand side of \eqref{4.10}
\end{rema}

Using \eqref{4.6}, \eqref{4.10} and H\"older estimates of solutions to elliptic equation,
we have the following lemma,
\begin{lemm}\label{le12}
For  $0<\alpha<1$ and any given symmetric tensor $q$, suppose $\vec{y}$ is any solution to \eqref{4.5}. Then, we have
the following estimate
\begin{eqnarray}
\|\vec{y}\|_{C^{2,\alpha}(\mathbb{S}^2)}\leq C(\|q\|_{C^{1,\alpha}(\mathbb{S}^2)}+\|q_{12,12}-q_{11,22}-q_{22,11}+q_{12,21}\|_{C^{\alpha}(\mathbb{S}^2)}),\nonumber
\end{eqnarray}
where $C$ is constant depending  on $\alpha$ and $\Sigma$.
\end{lemm}

For any vector field $\vec{z}$ on the embedded surface $\vec{r}$, we take $q=q(\vec{z},\nabla \vec{z})$ defined by \eqref{4.4}. Since the term $q_{12,12}-q_{11,22}-q_{22,11}+q_{12,21}$ is the linearized Gauss curvature, it does not include any third order derivatives of $\vec{z}$.  By \eqref{4.4}, in every term of $q_{ij}$ except the term $\tilde{g}_{ij}-g_{ij}$, the vector $\vec{z}$ appears
at least two times. Thus by the definition of map $\Phi$, Lemma \ref{le12} implies
\begin{eqnarray}\label{estimate}
&&\|\Phi(\vec{z})\|_{C^{2,\alpha}(\mathbb{S}^2)}\\
&\leq& C_1(\|g_{ij}-\tilde{g}_{ij}\|_{C^{2,\alpha}(\mathbb{S}^2)}+\|\vec{z}\|_{C^{2,\alpha}(\mathbb{S}^2)}^{2}+\|\vec{z}\|_{C^{2,\alpha}(\mathbb{S}^2)}^{3}
+\|\vec{z}\|_{C^{2,\alpha}(\mathbb{S}^2)}^{4}).\nonumber
\end{eqnarray}

For any number $\delta<1,$ let $$B_{\delta}=\{\vec{z} \text { some vector field on the surface } \vec{r}; \|\vec{z}\|_{C^{2,\alpha}(\mathbb{S}^2)}\leq \delta\}.$$
If we take sufficient small $\epsilon$ and $\delta$,  the map $\Phi$ is well defined in $B_{\nu}$ by \eqref{estimate}.
We want to prove that the map $\Phi$ is a contraction map. Suppose $\vec{z}_1,\vec{z}_2 \in B_{\delta}$, and
$$q=q(\vec{z}_1,\nabla \vec{z}_1)-q(\vec{z}_2,\nabla \vec{z}_2).$$
By \eqref{4.4} we have

\begin{eqnarray}
&&q_{ij}\nonumber\\
&=&-\sigma_{\alpha\beta}(\vec{r})[((\vec{z}_1)^{\alpha}_i-(\vec{z}_2)^{\alpha}_i)(\vec{z}_1)^{\beta}_j+(\vec{z}_2)^{\alpha}_i((\vec{z}_1)^{\beta}_j-(\vec{z}_2)^{\beta}_j)]\nonumber\\
&&-[(\vec{z}_1-\vec{z}_2)^{\mu}(\int_{0}^{1}\frac{\partial F_{\alpha\beta\gamma\lambda}}{\partial \vec{y}^\mu}(t\vec{z}_1+(1-t)\vec{z}_2)dt)\vec{z}_1^{\gamma}\vec{z}_1^\lambda\nonumber\\
&&+F_{\alpha\beta\gamma\lambda}(\vec{z}_2)(\vec{z}_1+\vec{z}_2)^{\gamma}(\vec{z}_1-\vec{z}_2)^{\lambda}]\vec{r}^{\alpha}_{i}\vec{r}^{\beta}_{j}\nonumber\\
&&-[(\vec{z}_1-\vec{z}_2)^{\mu}(\int_{0}^{1}\frac{\partial G_{\alpha\beta\gamma}}{\partial \vec{y}^\mu}(t\vec{z}_1+(1-t)\vec{z}_2) dt)\vec{z}_1^{\gamma}
(\vec{r}_i^{\alpha}(\vec{z}_1)_j^{\beta}+\vec{r}_j^{\beta}(\vec{z}_1)_i^{\alpha}+(\vec{z}_1)_i^{\alpha}(\vec{z}_1)_j^{\beta})\nonumber\\
&&+G_{\alpha\beta\gamma}(\vec{z}_2)(\vec{z}_1-\vec{z}_2)^{\gamma}(\vec{r}_i^{\alpha}(\vec{z}_1)_j^{\beta}+\vec{r}_j^{\beta}(\vec{z}_1)_i^{\alpha}+(\vec{z}_1)_i^{\alpha}(\vec{z}_1)_j^{\beta})\nonumber\\
&&+G_{\alpha\beta\gamma}(\vec{z}_2)\vec{z}_2^{\gamma}
(\vec{r}_i^{\alpha}(\vec{z}_1-\vec{z}_2)_j^{\beta}+\vec{r}_j^{\beta}(\vec{z}_1-\vec{z}_2)_i^{\alpha}+(\vec{z}_1-\vec{z}_2)_i^{\alpha}(\vec{z}_1)_j^{\beta}+(\vec{z}_2)^{\alpha}_i(\vec{z}_1-\vec{z}_2)_j^{\beta}).
\nonumber
\end{eqnarray}
Hence by Lemma \ref{le12}, we have
$$\|\Phi(\vec{z}_1)-\Phi(\vec{z}_2)\|_{C^{2,\alpha}(\mathbb{S}^2)}\leq C\delta\|\vec{z}_1-\vec{z}_2\|_{C^{2,\alpha}(\mathbb{S}^2)}.$$
If we choose sufficient small $\delta$, then map $\Phi$ is a contraction map which implies that there exists one solution to \eqref{4.5}. We completes the proof of Theorem \ref{openness}.

A direct corollary of the openness theorem is the following openness in conformal flat space.
\begin{coro} Suppose $ds^2$ is the symmetric conformal flat metric in $\mathbb{R}^3$. Namely, there is some positive function $\phi(r)$ satisfying $(r\phi(r))'>0$, such that
$$ds^2=\phi^2(r)(dr^2+r^2dS_2^2).$$ Then for some metric $g$ on $2$ sphere, if it can be isometrically embedded into  $(\mathbb{R}^3,ds^2)$ as some strictly convex surface, then the perturbation metric $\tilde{g}$ of $g$ also can be isometrically embedded into the conformal space $(\mathbb{R}^3,ds^2)$ as another strictly convex surface.
\end{coro}
\begin{proof}
The metric can be rewrite to the warped product metric \eqref{metric}. By openness theorem, we obtain our result.
\end{proof}
\section{Discussion of the existence results}
Using the openness theorem, we can prove some existence theorem for AdS space. The metric of the AdS space is defined by \eqref{AdS}, so it is a asymptotic Euclidean or hyperbolic space.
\begin{theo}
For any metric $g$ on $\mathbb{S}^2$, if its Gaussian curvature $K>-\kappa$, then it can be isometrically embedded into AdS space  with $\kappa$ and any $m$.  More precisely, $( \mathbb{S}^2,g)$ always can be isometrically embedded as a surface outside some large ball $B_R$ in AdS space.
\end{theo}
\begin{proof}
Using the isometric embedding theorems in space form, we know that, the metric $g$ can be embedded into space form with the sectional curvature $\kappa$. We translate $\vec{r}$ to out of some sufficient large ball $B_R$ where $\vec{r}$ is the position vector of the embedded surface. Let  $g_R$ be the induced metric on $\vec{r},$ we have
$$\|g-g_R\|_{C^{3,\alpha}(\mathbb{S}^2)}\leq \frac{C}{R},$$ where $C$ is some constant not depending on $R$.
In what follows we will prove the second fundamental form of the surfaces $\vec{r}$ outside $B_R$ are uniformly bounded. Denote
 the second fundamental form of $\vec{r}$ in space form and AdS space outside $B_R$ by $h,h_R$ respectively. We also denote the unit normal and Levi-Civita connection with respect to space form and AdS space by $\nu^{\kappa},\nu$ and $\nabla^{\kappa},\nabla$. Meanwhile we assume the local coordinate of sphere $\{ x_1,x_2\},$ then
it is not very difficult to check that $$\|\nu-\nu^{\kappa}\|_{C^{2,\alpha}(\mathbb{S}^2)}\leq \frac{C}{R},\text{ and } \|\nabla^{\kappa}_i\vec{r}_j-\nabla_i \vec{r}_j\|_{C^{2,\alpha}(\mathbb{S}^2)}\leq \frac{C}{R}.$$ Hence we have
$$\| h-h_R\|_{C^{2,\alpha}(\mathbb{S}^2)}\leq \frac{C}{R},$$ where $C$ is still some constant independent of $R$.
Thus, the constant appears in the proof of the openness theorem \ref{openness} is uniformly  bounded. If we choose $R$ sufficient large, the map $\Phi$ still is a contraction map. Hence metric $g$ is in the $\epsilon$ neighborhood of some $g_R$ where $\epsilon$ is the constant appears in Theorem \ref{openness}. So by openness, there is some isometric embedding convex surface for metric $g$.
\end{proof}
For the sake of completeness, we rough review some existence results obtained by Guan- Lu \cite{GL1} and Lu \cite{Lu}.
In fact the above existence theorem is a corollary of the openness theorem . For further discussion about existence results, we need a priori estimates, i.e.  the estimate of the Gauss equation
$$\kappa_1\kappa_2=\frac{\bar{R}}{2}+K-\bar{Ric}(\nu,\nu).$$ where $K$ is the Gauss curvature. The following Heinz type $C^2$ a prior estimate has been found by Lu \cite{Lu}.
\begin{theo}\label{Lu}
Suppose $(\mathbb{S}^2,g)$ is isometrically embedded into some ambient Reimannian manifold $(N^3,\bar{g})$ such that it is null homologous. Let $X$ be some isometric embedding.
Assume that $$ R(x) - \bar{R}(X(x))+2 \inf\{Ric_{X(x)}(\mu, \mu):\mu \in T_{X(x)}U, |\mu| = 1\}\geq C_0>0,$$ for any $x\in \mathbb{S}^2$ where $C_0$ is a positive constant. Further assume that $\bar{R}\geq -6\kappa^2$ and $R>-2\kappa^2$ for some constant $\kappa$. Then the mean curvature $ H$ of the embedded surface is bounded from above,  $$H\leq C,$$ where  $C$  is a positive constant depending  on $C_0$, $\inf (R+2\kappa^2),\|g\|_{C^3}$ and $\|\bar{g}\|_{C^3}$, but not on the position of the embedded surface.
\end{theo}


Now, we have some existence result for  warped product space with no singularity. Assume that there is some universal constant $K_0$  such that \\
\begin{eqnarray}\label{assumption}
  \text{(a)} \lim\sup_{r\rightarrow \infty} \frac{1-f^2}{r^2}\leq K_0, \ \  \text{ (b)} \lim\sup_{r\rightarrow \infty}-\frac{1-f^2+2rff'}{r^2}\leq K_0.
 \end{eqnarray}
 With the above condition, there exists some constant $K_1$ depending on $f$ and $K_0$ such that
 \begin{eqnarray}\label{assumption1}
  \max\{\frac{1-f^2}{r^2}, -\frac{1-f^2+2rff'}{r^2}\}\leq K_1.
 \end{eqnarray}
 Hence we can brief review the proof of the Theorem \ref{theorem2} in \cite{GL1}, \cite{Lu} .\\

\noindent {\bf Proof of Theorem \ref{theorem2}:} We use continuity method to obtain the existence result. At first, we use normalized Ricci flow on $\mathbb{S}^2$ to create some path to connect the given metric $g$ and some standard metric with constant Gauss curvature  bigger than $K_0$. See \cite{H} for more detail.  The assumption (a) and (b) in \eqref{assumption} can preserve the strictly convex of the family of surface defined by the Ricci flow. The constant Gauss curvature surface always can be isometrically embedded in the warped product space as some small enough geodesic sphere if there is no singularity. Thus it remains to prove the openness and closeness. Openness has been obtained in Theorem \ref{openness}. For the $C^{2}$ estimates, we need to check the conditions of Theorem \ref{Lu}. Let $C_0=2(\min_{\mathbb{S}^2} K-K_1) ,$ and choose sufficient large $\kappa,$ we have $R=2K>-2\kappa^2.$  By \eqref{2.5} and \eqref{assumption}, we have
$$\bar{R}=2(\frac{2ff'}{r}+\frac{f^2-1}{r^2})\geq -6K_1.$$ Thus choosing sufficient large $\kappa$, we also have $\bar{R}>-6\kappa^2$. Therefore we obtain the $C^2$ a prior estimate independent of the position of the embedded surface.  The $C^1$ estimate directly comes from the isometric embedding system. It remians to establish $C^0$ estimate for completing the proof. The curvature estimate implies that the principal curvatures are bounded. Hence along the path, by definition of the variation $\tau$, we have
$$\vec{r}_t=\vec{r}_0+\int_0^t\tau dt. $$ Using the Schauder  estimate \eqref{4.6}, we have the uniform bound for $\tau$, which implies uniform bound for $\vec{r}_t$. We complete the proof of Theorem \ref{theorem2}.

\section{The infinitesimal rigidity in space form}

In this section, we will revisit the infinitesimal  rigidity of the isometric embedding problem for  convex hypersurfaces in space form.  At first, we reprove the infinitesimal  rigidity for convex surface in $\mathbb{R}^3$ only using the maximum principle. Then using Beltrami map, we can generalize to the space form cases.

As in section $3,$ the infinitesimal problem of the isometric embedding problem is to consider the following linear system: for any given convex surface $\vec{r}$,
\begin{eqnarray}\label{infi}
d\vec{r}\cdot d\tau=0.
\end{eqnarray}
Obviously, there is some vector $Y$ such that
\begin{eqnarray}\label{suppose}
d\tau=Y\times d\vec{r}.
\end{eqnarray}
 Differentiating  the above equation, we have
$$d^2\tau=dY\times dr=0,$$ which implies  $dY$ is parallel to the tangent plane. hence we can assume $$d\vec{r}=\vec{r}_idx^i, dY=Y_idx^i \text{ and } Y_i=W_{i}^{k}\vec{r}_k.$$ By $dY\times dr=0$ we have
\begin{eqnarray}
W_{i}^k\vec{r}_k\times \vec{r}_j dx^i\wedge dx^j=(W_1^k\vec{r}_k\times \vec{r}_2- W_2^k\vec{r}_k\times \vec{r}_1)dx^1\wedge dx^2=0,\nonumber
\end{eqnarray}
which implies
\begin{eqnarray}\label{sW}
W_1^1+W_2^2=0.
\end{eqnarray}
Now we use "," to denote covariant derivative. We have
\begin{eqnarray}\label{6.4}
Y_{i,j}=W_{i,j}^k\vec{r}_k+W_{i}^k\vec{r}_{k,j}=W_{i,j}^k\vec{r}_k+W_{i}^kh_{kj}\nu,
\end{eqnarray}
where $h_{ij}$ and $\nu$ are the second fundamental form and unit outer normal of the surface $\vec{r}$. Since $Y_{i,j}=Y_{j,i},$ by \eqref{6.4} we have the compatible equations
$$W_{i,j}^k=W_{j,i}^k \text{ and }  W_i^kh_{kj}=W_j^kh_{ki}.$$ The second equation is
$$-W_{1}^1h_{12}-W_{1}^2h_{22}+W_{2}^1h_{11}+W_{2}^2h_{21}=0.$$ Multiplied by $\det(h_{ij}),$ it becomes
\begin{eqnarray}\label{6.5}
W_{1}^1h^{21}-W_{1}^2h^{11}+W_{2}^1h^{22}-W_{2}^2h^{12}=0.
\end{eqnarray}
We introduce a new tensor $a=a_{ij}dx^idx^j$ defined by
\begin{eqnarray}\label{deofa}
a_{11}=-W_{1}^2, a_{12}=-W_{2}^2,a_{21}=W_{1}^1, a_{22}=W_{2}^1.
\end{eqnarray}
Then, by \eqref{sW} $a_{ij}$ is a symmetric matrix, and the above equation can be rewritten as
$$h^{ij}a_{ij}=0.$$
Taking covariant derivative, we have
\begin{eqnarray}
-h^{ij}_{,2}a_{ij}=h^{ij}W_{i,j}^1,  h^{ij}_{,1}a_{ij}=h^{ij}W_{i,j}^2.\nonumber
\end{eqnarray}
Hence we have
\begin{eqnarray}
h^{ij}Y_{i,j}=h^{ij}W_{i,j}^k\vec{r}_k+W_{i}^i\nu=-h^{ij}_{,2}a_{ij}\vec{r}_1+h^{ij}_{,1}a_{ij}\vec{r}_2.\nonumber
\end{eqnarray}
By the Lemma 4 in \cite{GWZ}, we have
$$a_{11}^2+a_{12}^2+a_{21}^2+a_{22}^2\leq -C\det(a_{ij}),$$ which means
$$(W_{1}^1)^2+(W_{1}^2)^2+(W_{2}^1)^2+(W_{2}^2)^2\leq -C\det(W_{ij}).$$
Hence, we conclude
$$h^{ij}Y_{i,j}=-h^{ij}_{,2}a_{ij}\frac{A_{1l}Y_l}{\det W}+h^{ij}_{,1}a_{ij}\frac{A_{2l}Y_l}{\det W},$$ where, $A_{ij}$ is the cofactor of $W_{ij}$. Now for any constant vector $E$ in $\mathbb{R}^3,$ taking inner product with $E$ in the previous equation, we obtain some elliptic partial differential equation of second order of $Y\cdot E$. Then the strong maximum principle tells us
$Y\cdot E$ is constant which  implies $Y$ is a constant vector.

Hence, we have the well known infinitesimal rigidity in Euclidean space.
\begin{prop}\label{15}
In three dimensional Euclidean space, the solutions to \eqref{infi} are
$$A\times \vec{r}+B,$$ where $A, B$ are two constant vectors.
\end{prop}
\begin{rema} We also can give anther a slight different proof of the above proposition, in which we  establish some elliptic  equation for the vector,
\begin{eqnarray}
\vec{b}=\tau-Y\times\vec{r},\nonumber
\end{eqnarray}
where $\tau$ and $Y$ is defined by \eqref{infi} and \eqref{suppose}. Then, we also use strong maximum principle to obtain that $\vec{b}$ is a constant.
\end{rema}
\begin{proof} Obviously, by \eqref{suppose} we have
\begin{eqnarray}\label{6.7}
d\vec{b}=-d Y\times\vec{r}.
\end{eqnarray}
We still write
$$\vec{b}_k=-Y_k\times \vec{r}, \ \  Y_k=W_k^l\vec{r}_l.$$ Then we have
\begin{eqnarray}\label{6.8}
\vec{b}_i&=&-Y_i\times\vec{r}=-W_{i}^k\vec{r}_k\times\vec{r}\\
\vec{b}_{i.j}&=&-W_{i,j}^{k}\vec{r}_k\times\vec{r}-h_{kj}W_{i}^k\nu\times\vec{r}-W_{i}^k\vec{r}_k\times\vec{r}_j\nonumber.
\end{eqnarray}
Multiplied by $h^{ij}$ it becomes
\begin{eqnarray}
h^{ij}\vec{b}_{i.j}
&=&-h^{ij}W_{i,j}^{k}\vec{r}_k\times\vec{r}-W^{k}_{k}\nu\times\vec{r}-h^{ij}W_{i}^k\vec{r}_k\times\vec{r}_j\nonumber\\
&=&-h^{ij}W_{i,j}^{k}\vec{r}_k\times\vec{r}\nonumber,
\end{eqnarray}
where we have used \eqref{sW} and \eqref{6.5} in the last equality. Using the first equation of \eqref{6.8}, we obtain some elliptic equation of $\vec{b}$,
$$h^{ij}\vec{b}_{i,j}=-h^{ij}_{,2}a_{ij}\frac{A_{1l}\vec{b}_l}{\det W}+h^{ij}_{,1}a_{ij}\frac{A_{2l}\vec{b}_l}{\det W},$$ where $A_{ij}$ is the cofactor of $W_{ij},$ and $a_{ij}$ is defined in \eqref{deofa}. The strong maximum principle implies that $\vec{b}$ is a constant.

\end{proof}

 The paper \cite{LW}  discusses the infinitesimal rigidity problem of the convex surfaces in hyperbolic three space. In fact, we can use Beltrami map to turn the infinitesimal rigidity problem in  space forms to the one in Euclidean space.

 The de sitter space time is a hyperboloid $-x_0^2+x_1^2+\cdots+x_n^2=-1$ in Minkowski space $\mathbb{R}^{1,n}.$ For infinitesimal rigidity problem, there are a family of  $\vec{r}_t$ satisfying $$\vec{r}_t\cdot \vec{r}_t=-1,\text{ and } d\vec{r}_t\cdot d\vec{r}_t= g, $$
 where $g$ is a given metric  on the original manifold and $\cdot$ denotes the Minkowski inner product. Obviously, Minkowski inner product restricted in $\mathbb{R}^n$ is Euclidean inner product. Set $\tau=\frac{\p \vec{r}_t}{\p t}|_{t=0},$ we have $$\tau\cdot \vec{r}=0, \text{ and } d\vec{r}\cdot d\tau =0.$$
 Set
 $$\tau= (\frac{1}{x_0}\sum_iA_ix_i,A_1,\cdots,A_n)$$ and
 $$\bar{r}=(x_1,\cdots,x_n), A=(A_1,\cdots, A_n),$$  we have $$\vec{r}=(x_0, \bar{r}), \text{ and } \tau=(\frac{A\ast \bar{r}}{x_0}, A),$$ where we use $\ast$ to denote the inner product of the Euclidean space. Hence we obtain
 \begin{eqnarray}
 d\vec{r}\cdot d\tau&=& (dx_0, d\bar{r})\cdot (\frac{dA\ast \bar{r}+A\ast d\bar{r}}{x_0}-A\ast \bar{r} \frac{dx_0}{x_0^2}, dA)\nonumber\\
 &=&\frac{1}{x_0^2}(x_0^2d\bar{r}\ast dA-x_0dA\ast\bar{r}dx_0-x_0A\ast d\bar{r}dx_0+A\ast \bar{r}(dx_0)^2)\nonumber\\
 &=&x_0^2\frac{x_0d\bar{r}-\bar{r}dx_0}{x_0^2}\ast\frac{x_0dA-Adx_0}{x_0^2}\nonumber\\
 &=&x_0^2d\frac{A}{x_0}\ast d\frac{\bar{r}}{x_0}=0\nonumber.
 \end{eqnarray}
 Thus, we have
 \begin{eqnarray}\label{Es}
 d(\frac{\bar{r}}{x_0})\ast d(\frac{A}{x_0})=0.
 \end{eqnarray}
Now let's study the hypersurface $\tilde{r}=\bar{r}/x_0$. We believe the following Lemma is well known, but we do not find appropriate reference. Hence, we include a short proof here.
\begin{lemm}
The second fundamental form of the hypersurface $\tilde{r}$ is conformal to the second fundamental form of $\vec{r}$ in the normal direction in hyperbolic space and the conformal function is always positive or negative.
\end{lemm}
\begin{proof}
The normal direction of $\vec{r}$ in hyperboloid is
$$\frac{1}{\sqrt{|g|}}\left|\begin{matrix}e_0&e_1&\cdots&e_n\\
-(x_0)_1& &\bar{r}_1\\
\cdots&&\cdots\\
-(x_0)_{n-1}& & \bar{r}_{n-1}\\
-x_0& & \bar{r}
\end{matrix}\right|.$$
Here, $e_i=\frac{\partial}{\partial x_i},0\leq i\leq n$ and $g$ is the first fundamental form of $\vec{r}$. Hence the second fundamental form for $\vec{r}$ in hyperbolic space is
$$\frac{1}{\sqrt{|g|}}\left|\begin{matrix}-(x_0)_{ij}&&\bar{r}_{ij}\\
-(x_0)_1& &\bar{r}_1\\
\cdots&&\cdots\\
-(x_0)_{n-1}& & \bar{r}_{n-1}\\
-x_0& & \bar{r}
\end{matrix}\right|=\frac{(-1)^{n+1}x_0}{\sqrt{|g|}}\left|\begin{matrix}\bar{r}_{ij}-\frac{(x_0)_{ij}}{x_0}\bar{r}\\
\bar{r}_1-\frac{(x_0)_1}{x_0}\bar{r}\\
\cdots\\
\bar{r}_{n-1}-\frac{(x_0)_{n-1}}{x_0}\bar{r}\end{matrix}\right|.$$
On the other hand, the normal of the hypersurface $\tilde{r}$ is
$$\frac{1}{\sqrt{|\tilde{g}|}}\left|\begin{matrix}e_1&\cdots&e_n\\
&\frac{\bar{r}_1}{x_0}-\frac{(x_0)_1}{x_0^2}\bar{r}\\
&\cdots\\
&\frac{ \bar{r}_{n-1}}{x_0}-\frac{(x_0)_{n-1}}{x_0^2}\bar{r}
\end{matrix}\right|,$$
where $\tilde{g}$ is the first fundamental form of $\tilde{r}$.
Note that $$\tilde{r}_{ij}=\frac{\bar{r}_{ij}}{x_0}-\frac{(x_0)_{ij}\bar{r}}{x_0^2}-\frac{(x_0)_j\tilde{r}_i+(x_0)_i\tilde{r}_j}{x_0}.$$
A direct calculation of $\tilde{h}_{ij}=\tilde{r}_{ij}\ast \tilde{n}$ immediately yields the two fundamental forms are conformal, and moreover the conformal function is $(-x_0)^{n+1}\frac{\sqrt{|\tilde{g}|}}{\sqrt{|g|}}.$
\end{proof}

Using these preliminaries, we can prove  the Theorem \ref{rigidity}.\\

\begin{pfthmsl}
It's divided into three cases. For Euclidean space, Proposition \ref{15} says that $\tau,$ the solutions to \eqref{infi}, are generated by the Lie algebra of the Lie group $O(3)\times \mathbb{R}^3$ which is the isometry group of the three dimensional Euclidean space.

For the hyperbolic space case, the previous discussion shows that we have two constant vectors $Y,Z$ in $\mathbb{R}^3$, such that
$$A=Y\times_E \bar{r}+x_0 Z,$$ where $\times_E$ is the cross product of Euclidean space. Hence, we have
$$\tau=(x_0Z\ast\bar{r}, Y\times_E \bar{r}+x_0Z).$$ Denote the Beltrami map by $p,$ the pre image of $\frac{\p}{\p x_i}$ is $E_i=\frac{x_i}{x_0}\frac{\p}{\p x_0}+\frac{\p}{\p x_i}$ since the coordinate of $\mathbb{R} ^3$ is $x_1,x_2,x_3.$ Let
$X=p^*\bar{r}, Y_0=p^*Y, Z_0=p^*Z$. It is easy to check that
$$\tau=x_0Y_0\times X+x_0Z_0,$$ where $\times$ is the cross product in hyperbolic space. The result was proved by Lin-Wang \cite{LW} at first. It is also obvious that $\tau$ is generated by the Lie algebra of the isometry group of Minkowski space $\mathbb{R}^{1,3}$ restricted to hyperboloid, see \cite{K} for more detail.

The third case is sphere space. It is well known that any strictly convex surface in three sphere is included in some hemisphere by Bonnet-Myers theorem \cite{M}.  We pose the radius $1/\kappa$ three sphere in  $\mathbb{R}^4$. Let its  coordinate be $x_0,x_1,x_2,x_3$. Without loss of generality, assume the ambient space is the upper hemisphere, namely  $x_0\geq 0$.  Using the Beltami map, we project the hemisphere to its equator hyperplane $\{x_1,x_2,x_3\}$.  Then, using the parameters $x_1,x_2,x_3$, the metric  can be rewritten by \eqref{metric}  with the warped function $f=\sqrt{1-\kappa r^2}$. A similar argument in  hyperbolic case shows the solutions to the infinitesimal problem also are
$$\tau=x_0Y_0\times X+x_0Z_0.$$ The meaning of $\times, X, Y_0,Z_0$ are similar except $x_0=\sqrt{1-x_1^2-x_2^2-x_3^2}$. The solutions also are generated by the Lie algebra of the Lie group $O(4)$, which is the isometry group of the three sphere. We complete the proof.
\end{pfthmsl}

\begin{rema}
The previous discussion also can be generalized into the hypersurface in high dimensional space form. Since we have more algebraic relations exhibited in Gauss equations, the convexity can be weakened to the condition that the rank of the second fundamental form matrix is equal to or great than two.  This result was proved by Dajczer-Rodriguez in \cite{DR} at first. More material about the infinitesimal rigidity for hypersurface can be found in \cite{D}.
\end{rema}
At the end of the section, we would like to point out another fact which is very natural.
\begin{rema}
In space form, the linear system $$d\vec{r}\cdot D\tau=q $$ can be rewritten  as inhomogeneous type of \eqref{Es}
$$d\frac{\bar{r}}{x_0}\ast d\frac{A}{x_0}=\frac{q}{x_0^2}.$$ Hence the solvability of the former linear system in Euclidean space implies the one in space form. The Euclidean space case has been done in \cite{N1}. As a direct corollary, the linear system  is always solvable in space form.

\end{rema}
\section{The non infinitesimal rigidity and non rigidity }
An important corollary of the infinitesimal rigidity in the space form is that we can calculate the dimension of the kernel of the linearized equation \eqref{3.3} discussed in section $3$.
\begin{theo}
Suppose $\vec{r}$ is a strictly convex surface in any three dimensional warped product space. The dimension of the solution space to the linear homogeneous equation
\begin{eqnarray}\label{ID}
d\vec{r}\cdot D\tau=0,\end{eqnarray} is always six, where $\tau$ is a smooth vector field in warped product space defined on sphere.
\end{theo}
\begin{proof}
In section $3$, we have shown that the dimension of the solution space to  \eqref{ID} is the dimension of the $\ker L_h$ defined by \eqref{L_h}. By the discussion of the previous section, we know that it is six in space form. Since the operator $L_h$ is elliptic defined on the cotangent bundle on sphere for strictly convex surface $\vec{r},$ it is a Fredholm operator. Its index is defined by
$$\text{ind} (L_h)=\text{dim}\ker L_h-\text{dim coker}L_h. $$ In section $3,$ we have proved Theorem \ref{LE} which implies
$\text{coker} L_h=0$ for any three dimensional warped product space.
Hence we have $$\text{ind}(L_h)=\text{dim}\ker L_h.$$ It is well known that the index only depends on the principal symbol of the operator. Furthermore, for a family of operators $L(t),$ let $a(t)$ be the family of the principal symbols. Suppose  $a(t)$ is invertible, denote its inverse by $b(t)$ and  $a(t), b(t)$ both are uniformly bounded, then the index of these operators are invariant. See
 \cite{HIII} for more detail.  In our case, by \eqref{L_h} the principal symbol of $L_h$ only depends on the second fundamental form $h$.
 Hence  for any given $h,$ we always can choose a path,
 $$h^t=(1-t)h+t g_0,$$ where $0\leq t\leq 1$ and $g_0$  is the canonical metric on sphere  whose Gauss curvature is constant $1$. In view of \eqref{3.4}, the principal symbol of the operators defined by $h^t$ are uniformly bounded. The inverse of these principal symbol  are
\begin{eqnarray}\label{7.2}
u_1&=&\frac{1}{\eta}[(q_{11}h^t_{22}-q_{22}h^t_{11})\xi_1-2(h^t_{12}q_{11}-h^t_{11}q_{12})\xi_2]\\
u_2&=&\frac{1}{\eta}[(q_{22}h^t_{11}-q_{11}h^t_{22})\xi_2-2(h^t_{12}q_{22}-h^t_{22}q_{12})\xi_1]\nonumber\\
\eta&=&h^t_{22}\xi_1^2+h^t_{11}\xi_2^2-2h^t_{12}\xi_1\xi_2\nonumber,
\end{eqnarray}
where $\xi$ is some unit cotangent vector and $q$ is symmetric $(0,2)$ tensor with $tr_{h^t} q=0 $. Obviously, the inverse principal symbols \eqref{7.2} are also uniformly bounded. Thus the index of the operator $L_h$ is same as $L_{g_0}.$ We can view the $L_{g_0}$  as the unit sphere embedded in Euclidean space. Hence  by the infinitesimal rigidity discussed in the previous section, $\text{ ind }(L_{g_0})=6,$  which implies
$$\text{dim}\ker L_h=6.$$
\end{proof}
A corollary of the above theorem is the following non infinitesimal rigidity.
\begin{coro}  If the sectional curvature of three dimensional warped product space is not constant, namely
$$\frac{f(r)f'(r)}{r}+\frac{1-f^2(r)}{r^2}\neq 0,$$
 for any embedded strictly convex surface $\vec{r},$ it isn't infinitesimally rigid.
\end{coro}
\begin{proof}
There are some well known facts about the isometry group of any $n$ Reimannian manifold. It is a Lie group and its dimension is at most  $n(n+1)/2.$ If the dimension achieves $n(n+1)/2$, it should be some constant sectional curvature space. See \cite{K} and \cite{Ei} for more explanation. In our case, since our three dimensional warped space is not space form, the dimension of the isometry group is less than six which implies the dimension of its Lie algebra is also less than six. The discussion in section $3$ tells us that for any $A$ in the Lie algebra,  $\tau$ generated by $A$ is the solution to \eqref{ID}.  But the previous theorem tell us that the kernel of the linear equation \eqref{ID} is always six. Hence there is some solution to \eqref{ID} which doesn't come from the Lie algebra, namely it's nontrivial solution. Hence convex surface is not infinitesimally rigid any more.
 \end{proof}

In view of the proof of the openness, we can use the previous corollary to obtain the non rigidity for any strictly convex surface.
\begin{theo}
In any three dimensional warped product space which is not a  space form,
for any embedded strictly convex surface, it isn't rigid.\end{theo}
\begin{proof}By the previous corollary, there exists some vector field $T$, which is the solution of \eqref{ID}, such that there is no action $A$ in the Lie algebra of the isometry group  satisfying $T=A\vec{r}$. In what follows, we will generate a family of isometric strictly convex surfaces $\vec{r}_{\epsilon}$ by the vector field $T.$ Let
 $$\vec{r}_{\epsilon}=\vec{r}+\epsilon T+ \epsilon^2 T_{\epsilon}$$ where $T_{\epsilon}$ is some undetermined vector field. As required,  $\vec{r}_{\epsilon}$ satisfy isometric system
$$d\vec{r}_{\epsilon}\cdot d\vec{r}_{\epsilon}=d\vec{r}\cdot d\vec{r}.$$

As discussed in section $4$,  we have a system similar to \eqref{4.5},
$$d\vec{r}\cdot D(\epsilon T+\epsilon^2 T_{\epsilon})=q(\epsilon T+\epsilon^2 T_{\epsilon},\nabla(\epsilon T+\epsilon^2 T_{\epsilon})),$$ which implies
$$d\vec{r}\cdot D  T_{\epsilon}=\frac{1}{\epsilon^2}q(\epsilon T+\epsilon^2 T_{\epsilon},\nabla(\epsilon T+\epsilon^2 T_{\epsilon})).$$
We will repeat the argument in the proof of the openness to find the solution $T_{\epsilon}$ to the above nonlinear system. We define some map using the previous system,
\begin{eqnarray}
F: C^{2,\alpha}(\mathbb{S}^2)&\mapsto &C^{2,\alpha}(\mathbb{S}^2)\nonumber,\\
T_{\epsilon}&\mapsto &T^*_{\epsilon}\nonumber,
\end{eqnarray}
where $T^*_{\epsilon}$ is the solution to
$$d\vec{r}\cdot D  T^*_{\epsilon}=\frac{1}{\epsilon^2}q(\epsilon T+\epsilon^2 T_{\epsilon},\nabla(\epsilon T+\epsilon^2 T_{\epsilon})),$$
and perpendicular to the kernel of \eqref{3.3}.
By Theorem \ref{LE}, the previous system is solvable and then we can find unique solution $T^*_{\epsilon}$ which is perpendicular to the kernel of \eqref{3.3},
thus the map $F$ is well defined.
Then, by \eqref{estimate} we have
\begin{eqnarray}\label{7.3}
&&\|F(T_{\epsilon})\|_{C^{2,\alpha}(\mathbb{S}^2)}\\
&\leq& C_1(\| T+\epsilon T_{\epsilon}\|_{C^{2,\alpha}(\mathbb{S}^2)}^{2}+\epsilon\|T+\epsilon T_{\epsilon}\|_{C^{2,\alpha}(\mathbb{S}^2)}^{3}
+\epsilon^2\|T+\epsilon T_{\epsilon}\|_{C^{2,\alpha}(\mathbb{S}^2)}^{4}).\nonumber
\end{eqnarray}
We  let $C_2=2C_1\|T\|^2_{C^{2,\alpha}(\mathbb{S}^2)}$ and use the notation $B_{C_2}$ defined in section $4$. The previous estimates tell us that the map $F$ is  from $B_{C_2}$ to $B_{C_2}$ for sufficient small $\epsilon$.  Using the explicit expression of $q_{ij}$ in section $4$, we also have
 $$\|F(T^1_{\epsilon})-F(T^2_{\epsilon})\|_{C^{2,\alpha}(\mathbb{S}^2)}\leq C\epsilon\|T^1_{\epsilon}-T^2_{\epsilon}\|_{C^{2,\alpha}(\mathbb{S}^2)},$$
 for $T^1_{\epsilon},T^2_{\epsilon}$ in $B_{C_2}$. If we choose sufficient small $\epsilon$, the map $F$ is a contraction map. Then, by contraction mapping theorem, we obtain the family of strictly convex surfaces $\vec{r}_{\epsilon}$.

We claim there is some $\epsilon_0$ such that $\vec{r}_{\epsilon_0}$ can not move to the surface $\vec{r}$ by some isometry of the ambient space. If it is not true, we can find a family of isometry $g_{\epsilon}$ of the ambient space such that
$$\vec{r}_{\epsilon}=g_{\epsilon}\vec{r}. $$ Then we can take derivative with respect to $\epsilon,$ and letting $\epsilon=0$, we find that the  vector field $T$ is an infinitesimal isometry vector field, which contradicts to our assumption. Thus $\vec{r}_{\epsilon_0}$ is isometric to the surface $\vec{r}$ but cannot move to $\vec{r}$ by some isometry of the ambient space.
\end{proof}

Maybe, we can propose another type of rigidity. At first, let us see some example. In \eqref{metric}, if we choose  the following warped function in $\mathbb{R}^3$,
\begin{eqnarray}\label{DFF}
f(r)=\left\{\begin{matrix}1,\text { if }0\leq r\leq r_0\\
g(r),  \text{       if }r>r_0 \end{matrix}\right.,
\end{eqnarray}
where $r_0$ is some given positive number and $g(r)$ is arbitrary positive function. Then if there is some convex surface contained in the ball $B_{r_0}$, the translation of the surface in the ball will not change the shape of the surface. But if the function $g$ is non trivial, the warped product space defined by \eqref{DFF} is not a space form. Hence, by the previous theorem the convex surface in $B_{r_0} $ is still not rigid.

If the requirement is  more detail, we can find some counterexample that the perturbation surface do not have the same shape, namely same second fundamental form. We will discuss it in the next section.

\section{The non rigidity for spheres}
In this section, we will try to find some counterexample for the isometric embedding problem in any dimensional space. More precisely, we will find some convex body in $n$ dimensional warped product space which is isometric to  unit slice sphere, but its second fundamental form is not same as the slice sphere.  The method we will apply here is to follow similar tricks  in the openness. Hence we need to discuss the linearized equation
\begin{eqnarray}\label{8.1}
2d\vec{r}\cdot D\tau=q,
\end{eqnarray}
in $\mathbb{R}^n$ equipped with the metric $$ds^2=\frac{1}{f^2(r)}dr^2+r^2dS_{n-1}.$$
Similar to Section $3$, we can rewrite this system as
\begin{eqnarray}\label{8.2}
v_{i,j}+v_{j,i}=2\phi h_{ij}+q_{ij},
\end{eqnarray}
where $\omega=\sum_i v_i dx^i, v_i=\tau\cdot \vec{r}_i$ is the $1$-form on sphere, $\phi=\tau\cdot \nu$ and $q$ is a symmetric $(0,2)$ tensor on sphere. By \eqref{sphere}, we have
$$\phi h_{ij}=\frac{f^2(r)}{r}\tau\cdot  \frac{\p}{\p r}g_{ij}=\frac{1}{r}\tau\cdot_E\frac{\p}{\p r}g_{ij},$$
where $\cdot_E$ is the inner product of Euclidean space.
Let $\phi_E=\tau\cdot_E\nu_E$ and $h^E_{ij}$ be the second fundamental form of $r$ sphere in Euclidean space, then we have
$$u_{i,j}+u_{j,i}=2\phi_Eh^E_{ij}+q_{ij}.$$
That means  for any slice sphere the system \eqref{8.1} can be viewed in Euclidean space.  Now we write down it in polar coordinate.
Let $(u_1,u_2,\cdots, u_{n-1})$ be the sphere coordinate and $(r,\theta_1,\theta_2,\cdots,\theta_{n-1})$ be the polar coordinate in ambient space,
We present the radius $r$ sphere in warped product space by the map
\begin{eqnarray}
&&\vec{r}(u_1,\cdots,u_{n-1})\nonumber\\
&=&r(\cos u_1\cos u_2\cdots\cos u_{n-1},\cos u_1\cos u_2\cdots\cos u_{n-2}\sin u_{n-1},\cdots,\sin u_1).\nonumber
\end{eqnarray}
Hence $r$ is constant and $\theta_i=u_i$. It is obvious that
\begin{eqnarray}
\frac{\partial}{\partial \theta_i}=\frac{\partial}{\partial u_i}.\nonumber
\end{eqnarray}
By \eqref{8.1}, we have
\begin{eqnarray}\label{8.3}
\frac{\partial}{\partial u_i}\cdot_E \frac{\partial \tau}{\partial u_j}
+\frac{\partial }{\partial u_j }\cdot_E \frac{\partial\tau}{\partial u_i}&=&q_{ij}.
\end{eqnarray}
Define a symmetric tensor
\begin{eqnarray}\label{8.4}
&&q_{ij}(\vec{r},\vec{y})\\
&=&-\sigma_{\alpha\beta}(\vec{r})\vec{y}^{\alpha}_{i}\vec{y}^{\beta}_{j}
-F_{\alpha\beta\gamma\lambda}(\vec{r},\vec{y})\vec{r}^{\alpha}_{i}\vec{r}^{\beta}_{j}\vec{y}^{\gamma}\vec{y}^{\lambda}
-G_{\alpha\beta\gamma}(\vec{r},\vec{y})\vec{y}^{\gamma}(\vec{r}^{\alpha}_{i}\vec{y}^{\beta}_{j}
+\vec{y}^{\alpha}_{i}\vec{r}^{\beta}_{j}+\vec{y}^{\alpha}_{i}\vec{y}^{\beta}_{j}),\nonumber
\end{eqnarray}
where
$$
F_{\alpha\beta\gamma\lambda}(\vec{r},\vec{y})=\int_{0}^{1}(1-t)\partial^{2}_{\gamma\lambda}\sigma_{\alpha\beta}(\vec{r}+t\vec{y})dt$$
and
$$G_{\alpha\beta\gamma}(\vec{r},\vec{y})=\int_{0}^{1}\partial_{\gamma}\sigma_{\alpha\beta}(\vec{r}+t\vec{y})dt.$$
For equation \eqref{8.3}, we have the following property. Let the Euclidean coordinate in $\mathbb{R}^n$ be $\{ z^1,\cdots, z^n\},$
\begin{lemm}\label{lem10}
 Suppose $h(t)$ is some smooth function of one variable. If we let $$\vec{y}=h(\sin u_1)\frac{\p}{\p z^n} $$ on the radius $r$ sphere, the solution to  the equation \eqref{8.3} with $q_{ij}$ given by \eqref{8.4} can be chosen by
\begin{eqnarray}\label{8.5}
\tau=\tilde{h}(\sin u_1)\frac{\p}{\p z^n},
\end{eqnarray}
where $\tilde{h}$ is another one variable smooth function.
\end{lemm}
\begin{proof}
To calculate $q_{ij}$ in Euclidean coordinate, we rewrite warped product metric in the Euclidean coordinate
\begin{eqnarray}\label{8.6}
ds^2&=&\frac{1}{f^2}dr^2+r^2dS_{n-1}\\
&=&(\frac{1}{f^2}-1)dr^2+dr^2+r^2dS_{n-1}\nonumber\\
&=&\psi(r)z^{\alpha}z^{\beta}dz^{\alpha}dz^{\beta}+\delta_{\alpha\beta}dz^{\alpha}dz^{\beta},\nonumber
\end{eqnarray}
where we have
\begin{eqnarray}\label{8.7}
\psi(r)=\frac{1}{r^2}(\frac{1}{f^2(r)}-1), \text{ and }  r^2=\sum_i(z^i)^2.
\end{eqnarray}

We divide $q$ into  three terms  $$q=-I-II-III.$$ In what follows we will calculate $q_{ij}$ term by term.
At first, we have
$$I_{ij}=\sigma_{nn}(\vec{r})h_ih_j.$$
Note that,
except $i=j=1$, the above term is zero, and $I_{11}$ is
\begin{eqnarray}
I_{11}&=&(\psi(r)(z^n)^2+1)h^2_1\nonumber\\
&=&(r^2\psi(r)\sin^2u_1+1)(h')^2\cos^2u_1.\nonumber
\end{eqnarray}
The second term is
$$II_{ij}=\int^1_0(1-t)\p_{nn}\sigma_{\alpha\beta}(\vec{r}+t\vec{y})z_i^{\alpha}z^{\beta}_jh^2 dt.$$
Let $\tilde{r}=|\vec{r}+t\vec{y}|,$ then we have
\begin{eqnarray}\label{8.8}
\p_n\sigma_{\alpha\beta}(\vec{r}+t\vec{y})&=&\frac{\psi'(\tilde{r})}{\tilde{r}}(z^n+t h)(z^{\alpha}+t\vec{y}^{\alpha})(z^\beta+t\vec{y}^{\beta})\\
&&+\psi(\tilde{r})(\delta_{\alpha n}(z^{\beta}+t\vec{y}^{\beta})+\delta_{\beta n}(z^{\alpha}+t\vec{y}^{\beta})).\nonumber
\end{eqnarray}
Thus we also have
\begin{eqnarray}
&&\p_{nn}\sigma_{\alpha\beta}(\vec{r}+t\vec{y})\\
&=&[\frac{\psi'(\tilde{r})}{\tilde{r}}]'(z^n+t h)^2(z^{\alpha}+t\vec{y}^{\alpha})(z^\beta+t\vec{y}^{\beta})+\frac{\psi'(\tilde{r})}{\tilde{r}}(z^{\alpha}+t\vec{y}^{\alpha})(z^\beta+t\vec{y}^{\beta})\nonumber\\
&&+2\frac{\psi'(\tilde{r})}{\tilde{r}}(z^n+th)(\delta_{\alpha n}(z^{\beta}+t\vec{y}^{\beta})+\delta_{\beta n}(z^{\alpha}+t\vec{y}^{\alpha}))+2\psi(\tilde{r})\delta_{\alpha n}\delta_{\beta n}\nonumber.
\end{eqnarray}
It is obvious that
$$\tilde{r}^2=r^2+t^2(h(\sin u_1))^2+2tr\sin u_1h(\sin u_1)$$ which depends only on $\sin u_1$ and $t$.  Since
$\sum_{\alpha}(z^{\alpha})^2=r,$ for any $i=1,\cdots, n-1$, we have
$$\sum_{\alpha}z^{\alpha}z^{\alpha}_i=0.$$
Hence we have
\begin{eqnarray}
II_{ij}&=&\int^1_0(1-t)h^2 \{[\frac{\psi'(\tilde{r})}{\tilde{r}}]'(z^n+t h)^2+\frac{\psi'(\tilde{r})}{\tilde{r}}\}t^2(\vec{y}^n)^2z^n_i z^n_jdt\nonumber\\
&&+4\int_0^1(1-t)h^2\frac{\psi'(\tilde{r})}{\tilde{r}}(z^n+th)t\vec{y}^nz^n_iz^n_jdt+2\int^1_0(1-t)h^2\psi(\tilde{r})z^n_iz^n_j dt\nonumber.
\end{eqnarray}
Thus $II$ is zero except $i=j=1$. We also have
\begin{eqnarray}
II_{11}&=&\int^1_0(1-t)t^2h^4 \{[\frac{\psi'(\tilde{r})}{\tilde{r}}]'(r\sin u_1+t h)^2+\frac{\psi'(\tilde{r})}{\tilde{r}}\}r^2\cos^2 u_1dt\nonumber\\
&&+4\int_0^1(1-t)th^3\frac{\psi'(\tilde{r})}{\tilde{r}}(r\sin u_1+th)r^2\cos^2u_1dt+2\int^1_0(1-t)h^2\psi(\tilde{r})r^2\cos^2u_1 dt\nonumber.
\end{eqnarray}
For the $III$, by \eqref{8.8}, we have
\begin{eqnarray}
&&III_{ij}\\
&=&\int^1_0h\p_n\sigma_{\alpha\beta}(\vec{y}+t\vec{y})(z^{\alpha}_{i}\vec{y}^{\beta}_{j}
+\vec{y}^{\alpha}_{i}z^{\beta}_{j}+\vec{y}^{\alpha}_{i}\vec{y}^{\beta}_{j})dt\nonumber\\
&=&\int^1_0h\frac{\psi'(\tilde{r})}{\tilde{r}}(z^n+t h)[t\vec{y}^n(z^n_i\vec{y}^n_j+z^n_j\vec{y}^n_i)(z^n+t\vec{y}^{n})+(z^n+t\vec{y}^n)^2\vec{y}^n_i\vec{y}^n_j]dt\nonumber\\
&&+\int^1_0h\psi(\tilde{r})[(z^n_i\vec{y}^n_j+z^n_j\vec{y}^n_i)(z_n+th+t\vec{y}^n)+2\vec{y}^n_i\vec{y}^n_j(z^n+t\vec{y}^n)]dt.\nonumber
\end{eqnarray}
Hence  except $i=j=1$, $III$ is zero. For $III_{11}$, we have
\begin{eqnarray}
&&III_{11}\\
&=&\int^1_0h\frac{\psi'(\tilde{r})}{\tilde{r}}(r\sin u_1+t h)^2[2trhh'\cos^2u_1+(r\sin u_1+th)(h')^2\cos^2u_1]dt\nonumber\\
&&+\int^1_0h\psi(\tilde{r})[2rh'\cos^2u_1(r\sin u_1+2th)+2(h')^2\cos^2 u_1(r\sin u_1+th)]dt.\nonumber
\end{eqnarray}
Let
\begin{eqnarray}\label{W}
&&W(r,\sin u_1,h)\\
&=&(r^2\psi(r)\sin^2u_1+1)(h')^2+\int^1_0(1-t)t^2h^4 \{[\frac{\psi'(\tilde{r})}{\tilde{r}}]'(r\sin u_1+t h)^2+\frac{\psi'(\tilde{r})}{\tilde{r}}\}r^2dt\nonumber\\
&&+4\int_0^1(1-t)th^3\frac{\psi'(\tilde{r})}{\tilde{r}}(r\sin u_1+th)r^2dt+2\int^1_0(1-t)h^2\psi(\tilde{r})r^2 dt\nonumber.\\
&&+\int^1_0h\frac{\psi'(\tilde{r})}{\tilde{r}}(r\sin u_1+t h)^2[2trhh'+(r\sin u_1+th)(h')^2]dt\nonumber\\
&&+\int^1_0h\psi(\tilde{r})[2rh'(r\sin u_1+2th)+2(h')^2(r\sin u_1+th)]dt,\nonumber
\end{eqnarray}
we have
\begin{eqnarray}
q_{ij}=\left\{\begin{matrix}-W(r,\sin u_1,h)\cos^2u_1, &\text{ if}  i=j=1\\ 0,& \text{ otherwise }\end{matrix}\right..
\end{eqnarray}
We choose the solution in the form \eqref{8.5}, then the equation \eqref{8.3} becomes
$$2\cos u_1 \frac{d \tilde{h}}{d u_1}=-W\cos^2 u_1.$$  Hence if we take
\begin{eqnarray}\label{8.14}
\tilde{h}(t)=-\frac{1}{2}\int_0^t W(r,s,h(s))ds,
 \end{eqnarray}
then the solution is $\tilde{h}(\sin u_1)$.
\end{proof}
In what follows we will apply Banach contraction mapping theorem to obtain the solution. Define a map
\begin{eqnarray}\label{8.15}
T: C^k([-1,1])&\rightarrow &C^k([-1,1])\\  h&\mapsto & -\dfrac{1}{2\epsilon^2}\int_0^tW(r,s,\epsilon+\epsilon^2h(s))ds\nonumber.
\end{eqnarray}
It is a well defined map. Denote
$$m(s)=2\int^1_0(1-t)\psi(\tilde{r}(s,t))r^2 dt, \text{ and } C_0=\|m\|_{C^k([-1,1])}.$$ If $\|h\|_{C^k}\leq 2C_0$, let's estimate
$Th$,
\begin{eqnarray}\label{8.16}
&&\|Th\|_{C^k}\\
&\leq& C\|\frac{W(r,s,\epsilon+\epsilon^2h(s))}{\epsilon^2}\|_{C^{k-1}}\nonumber\\
&\leq&C\epsilon^2\|h\|^2_{C^k}+C\epsilon^2[\|1+\epsilon h\|^4_{C^{k-1}}+\epsilon^2\|1+\epsilon h\|^6_{C^{k-1}}]\nonumber\\
&&+C\epsilon[\|1+\epsilon h\|^3_{C^{k-1}}+\epsilon\|1+\epsilon h\|^4_{C^{k-1}}]+C_0\|1+\epsilon h\|^2_{C^{k-1}}\nonumber\\
&&+C\epsilon^2(1+\epsilon^2\|1+\epsilon h\|^2_{C^{k-1}})\|1+\epsilon h\|^2_{C^{k-1}}\|h\|_{C^k}\nonumber\\
&&+C\epsilon^2(1+\epsilon^2\|1+\epsilon h\|^2_{C^{k-1}})(1+\|1+\epsilon h\|_{C^{k-1}})\|1+\epsilon h\|_{C^{k-1}}\|h\|^2_{C^k}\nonumber\\
&&+C\epsilon(1+\|1+\epsilon h\|_{C^{k-1}})\|1+\epsilon h\|_{C^{k-1}}\|h\|_{C^k}(1+\epsilon^2\|h\|_{C^k})\nonumber\\
&\leq&C_0(1+\epsilon^2\|h\|^2_{C^k})+C\epsilon(1+\|h\|^6_{C^k})\nonumber\\
&\leq &2C_0\nonumber,
\end{eqnarray}
for sufficient small $\epsilon\leq \epsilon_0$. It remains to prove that $T$ is a contraction map. For $h_1,h_2$ satisfying
$\|h_1\|_{C^k},\|h_2\|_{C^k}\leq 2C_0,$  we have
\begin{eqnarray}\label{8.17}
&&\|Th_1-Th_2\|_{C^k}\\
&\leq& \frac{C}{\epsilon^2}\|W(r,s,\epsilon+\epsilon^2h_1(s))-W(r,s,\epsilon+\epsilon^2 h_2(s))\|_{C^{k-1}}\nonumber\\
&\leq&C(C_0)\epsilon\|h_1-h_2\|_{C^k}\nonumber.
\end{eqnarray}
Hence once we choose sufficient small $\epsilon_0$ such that, for $\epsilon\leq \epsilon_0$,
$$C(C_0)\epsilon\leq C(C_0)\epsilon_0<1.$$ Thus $T$ is a contraction map. If we choose $h_0=0$, the sequence generated by
$h_{n+1}=T(h_n)$ converges to some function $h^*(s)$  with $\|h^*\|_{C^k}\leq 2C_0.$
Let some  $C^k$ hypersurface $Y$  be
\begin{eqnarray}\label{surY}
Y=\vec{r}+(\epsilon+\epsilon^2h^*(\sin u_1))\frac{\p}{\p z^n}.
\end{eqnarray}
and $$\vec{y}=(\epsilon+\epsilon^2h^*(\sin u_1))\frac{\p}{\p z^n}.$$ By Lemma \ref{lem10}, it is obvious that
$$\frac{\p}{\p u_i}\cdot_E\frac{\p \vec{y}}{\p u_j}+\frac{\p}{\p u_j}\cdot_E\frac{\p \vec{y}}{\p u_i}=q_{ij}(\vec{r},\vec{y}).$$ Hence we have
$$d\vec{r}\cdot D\vec{y}=q(\vec{r},\vec{y}).$$
As discussed in section $3$, we have $$dY\cdot dY=d\vec{r}\cdot d \vec{r}.$$
\begin{rema}For two dimensional surface case, the equation \eqref{8.1} always can be solvable by the discussion in section $3$. But for high dimension case, the equation  \eqref{8.1} cannot always be solved.
\end{rema}
Now we are in the position to give the proof of Theorem \ref{NR} for non rigidity of geodesic spheres.\\

\noindent{\bf Proof of Theorem \ref{NR}.}
The above calculation implies that the $C^k$ hypersurface defined by \eqref{surY} is isometric to the $r$ sphere. In fact, we can show that it is smooth. By \eqref{2.6} the Gauss equations for the hypersurfce is
$$\sigma_2(h_{ij})=\frac{(n-1)(n-2)}{2r^{n-1}}+\frac{\bar{R}}{2}-\bar{Ric}(\nu,\nu).$$ Using the orthonormal frame $\{u_1,\cdots,u_{n-1}\}$, the second fundamental form is
$$h_{ij}=\frac{1}{f\varphi}(-\rho_{i,j}+\frac{f_{\rho}}{f}\rho_i\rho_j+f^2g_{ij}),$$ where $\varphi=X\cdot \nu$ is support function. We can assume the hypersurface is at least $C^3$. By the standard regularity theory of the solutions to linear elliptic PDEs, we have the smoothness of $\rho$ which implies the smoothness of $h_{ij}$. Since we have $$Y_{i,j}=\Gamma_{ij}^kY_k+h_{ij}\nu,$$ we get the smoothness of $Y$.

Let's calculate the second fundamental form of the hypersurface defined by \eqref{surY}. Note that
$$\frac{\p}{\p z^n}=\sin\theta_1\frac{\p}{\p r}+\frac{\cos \theta_1}{r}\frac{\p}{\p \theta_1},\vec{r}=r\frac{\p}{\p r},$$ we have
\begin{eqnarray}\label{8.19}
&&\\
2\rho(Y)&=&Y\cdot_E Y\nonumber\\
&=&[r+(\epsilon+\epsilon^2h^*(\sin u_1))\sin u_1]^2+(\epsilon+\epsilon^2h^*(\sin u_1))^2\cos^2u_1\nonumber\\
&=&r^2+2r\epsilon \sin u_1+O(\epsilon^2),\nonumber
\end{eqnarray}
hence we have
$$|Y|=\sqrt{r^2+2r\epsilon \sin u_1+O(\epsilon^2)}=r+\epsilon \sin u_1+O(\epsilon^2).$$
Since $h^*$ is estimated independent of $\epsilon$, we can choose $\epsilon$ sufficient small such that the first order term of $\epsilon$ is dominate.
As known,
$$\rho_{i,j}=\p^2_{u_iu_j}\rho-(D_{\frac{\p}{\p u_i}}\frac{\p}{\p u_j})\rho,$$ we have
\begin{eqnarray}
&& \rho_{1,1}=-r\epsilon \sin u_1+O(\epsilon^2);\ \   \rho_{i,j}=O(\epsilon^2) \text{ for } i\neq j; \\
&& \rho_{i,i}=-r\epsilon \sin u_1\Pi_{j=1}^{i-1}\cos^2u_j+O(\epsilon^2) \text{ for } i\neq 1\nonumber.
\end{eqnarray}
A direct computation shows
\begin{eqnarray}
\rho_i\rho_j=O(\epsilon^2); |\nabla \rho|^2=O(\epsilon^2); \varphi=\sqrt{2\rho-\frac{|\nabla \rho|^2}{f^2(|Y|)}}=r+\epsilon\sin u_1+O(\epsilon^2).\nonumber
\end{eqnarray}
Taylor expansion implies
\begin{eqnarray}\label{8.20}
&&f(|Y|)=f(r)+f'(r)\epsilon\sin u_1+O(\epsilon^2),\nonumber\\
&&f^2(|Y|)=f^2(r)+2f(r)f'(r)\epsilon \sin u_1+O(\epsilon^2)\nonumber,\\
&& uf(|Y|)=rf(r)+(rf'(r)+f(r))\epsilon\sin u_1+O(\epsilon^2)\nonumber.
\end{eqnarray}
For $i\neq j$, we have $h_{ij}=O(\epsilon^2)$. For the diagonal term, we have
\begin{eqnarray}
h_{i}^i&=&g^{ii}h_{ii}=\frac{f^2(r)+(2f(r)f'(r)+\frac{1}{r})\epsilon\sin u_1+O(\epsilon^2)}{rf(r)+(rf'(r)+f(r))\epsilon\sin u_1+O(\epsilon^2)}\nonumber\\
&=&\frac{f(r)}{r}+(\frac{f(r)f'(r)}{r}+\frac{1-f^2(r)}{r^2})\frac{\epsilon\sin u_1}{f(r)}+O(\epsilon^2)\nonumber.
\end{eqnarray}
For sufficient small $\epsilon$, we see that the above principal curvature is not same as the slice sphere case. If every sphere is rigid, the warped function should satisfy
\begin{eqnarray}\label{sf}
\frac{f(r)f'(r)}{r}+\frac{1-f^2(r)}{r^2}=0,
\end{eqnarray}
for any $r.$ It is well known that the solutions to the above equation are only the warped functions defining space forms. We complete our proof. \\

A direct corollary of the above theorem is that A. Ros's type theorem \cite{R} is not true in general warped product space, namely
\begin{theo}
In $n$ dimensional warped product space, constant scalar curvature hypersurface is not always some round sphere. More precisely, if
 \eqref{sf} does not hold for some $r,$ there exists some perturbation convex body of such $r$ radius slice sphere with the same constant scalar curvature but different second fundamental form.
\end{theo}
\begin{proof}
 The example in the theorem \ref{NR}  is isometric to the slice sphere, hence
their scalar curvatures are same, since scalar curvature only determined by metric. \end{proof}
\begin{rema}
We also can prove that $Y$ defined by \eqref{surY} has no same second fundamental form as a slice sphere in some cases by S. Brendle's theorem \cite{B2}. If $Y$ has  same second fundamental form as a slice sphere, its mean curvature should be constant, hence it must be a slice sphere by \cite{B2}. Obviously, $Y$ is not slice sphere.
\end{rema}
The uniqueness of the isometric embedding convex hypersurfaces are true in space form, \cite{GS}, \cite{D}. For the high dimensional case, in fact, it's an algebraic result. It can be roughly viewed in the following. Assume $\{e_1,\cdots,e_{n-1}\}$ are orthonormal frame on sphere and moreover diagonalize the second fundamental form, Gauss equation says, for $i\neq j$, $$\lambda_i\lambda_j=c+\bar{R}_{ijij},$$  where $\lambda_i$ is the principal curvature and  $R$ is the curvature tensor. Hence, for three distinct indices $i,j,k$, we have
$$\lambda_i=\sqrt{\frac{(c+\bar{R}_{ijij})(c+\bar{R}_{ikik})}{c+\bar{R}_{jkjk}}}.$$ Hence $\lambda_i$ is totally determined by metric and then we have the uniqueness. The uniqueness in  high dimensional case is a local property, since we have more Gauss equations. But Theorem \ref{NR} tells us that these more algebraic relations are not helpful in warped product space. It is a little surprising to us.

\section{ A rigidity theorem for sphere}
In the last section  some counterexample is constructed for the rigidity problem of isometric embedding. In this section, we try to recover the rigidity of spheres under the condition \eqref{Cond} which exhibits rotational symmetry and restricts "moving" directions hence it is posed to fit the isometry of the ambient space.

 For any vector fields $X,Y$ on the manifold $\Sigma$, we also define its Reimannian curvature tensor,
 $$R(X,Y)=\nabla_X\nabla_Y-\nabla_Y\nabla_X-\nabla_{[X,Y]}.$$  Suppose $\{e_i\}_{i=1}^{n-1}$ are orthonormal frame on $\Sigma,$  and $e_n$ is its unit outward normal. Let
 $$R_{ijkl}=R(e_i,e_j)e_k\cdot e_l,$$ then the Ricci curvature and scalar curvature are
 $$R_{ij}=g^{lk}R_{ilkj}=\sum_kR_{ikkj},\ \  R=g^{ij}R_{ij}=\sum_iR_{ii}.$$ The Einstein condition means
 \begin{eqnarray}\label{EM}
 R_{ij}=(n-2)g_{ij}=(n-2)\delta_{ij}.
 \end{eqnarray}

At first, we need some Poinc\'are type inequality. By Linchnerowicz's theorem \cite{Ob}, for Einstein metric defined by \eqref{EM},
 the first eigenvalue $\lambda_1$ of  the Laplacian with respect to the metric $g$ on manifold $\Sigma$ is not less than $n-1$.  Hence, for any smooth function $\chi$ on  $\Sigma$ satisfying
$$\int_{\Sigma}\chi d\sigma=0,$$ where $d\sigma$ is the volume element of $\Sigma,$ we have the following Poinc\'are type inequality
$$(n-1)\int_{\Sigma}\chi^2 d\sigma\leq \lambda_1\int_{\Sigma}\chi^2 d\sigma\leq \int_{\Sigma}|\nabla \chi|^2 d\sigma. $$

We still use $\ast$ to denote the inner product of Euclidean space. Let $E$ be some constant unit vector field, then \eqref{Cond} becomes
$$\int_{\Sigma}\vec{r}\ast E d\sigma=0. $$
If the metric on $\Sigma$ is Einstein, by Poinc\'are type inequality we have
 $$\int_{\Sigma}|\vec{r}\ast E|^2d\sigma\leq \frac{1}{n-1}\int_{\Sigma} \sum_i|e_i(\vec{r}\ast E)|^2d\sigma=\frac{1}{n-1}\int_{\Sigma} \sum_i|e_i\ast E|^2d\sigma.$$ Let $\{z^1,z^2,\cdots, z^n\}$ be the Euclidean coordinate for $\mathbb{R}^n$, then we have,
 \begin{eqnarray}\label{9.2}
 \int_{\Sigma}r^2d\sigma=\int_{\Sigma}\sum_j|\vec{r}\ast \frac{\p}{\p z^j}|^2d\sigma\leq \frac{1}{n-1}\int_{\Sigma} \sum_i|e_i|^2_Ed\sigma.
 \end{eqnarray}
 Let's calculate the Euclidean norm of $e_i$ using polar coordinate. Suppose $\{ E_1, E_2, \cdots, $ $E_n\}$ are the orthonormal frame of the warped metric defined in section $2$. We let
 $$\tilde{E}_1=\frac{E_1}{f}, \tilde{E}_2=E_2,\cdots \tilde{E}_n=E_n,$$ be some orthonormal frame with respect to Euclidean metric. We present the vector $e_i$ in the two different frames with the coefficient $a_i^m$ and $\tilde{a}_i^m$,
 $$e_i=\sum_ma^m_iE_m=\sum_m\tilde{a}^m_i\tilde{E}_m.$$ Hence we have
 $$\tilde{a}^1_i=fa^1_i, \tilde{a}_i^k=a^k_i, \text{ for } k=2,\cdots, n.$$ Since $e_i$ is unit with respect to the warped product metric, we have $\sum_m(a^m_i)^2=1$. Hence we obtain
 $$|e_i|_E^2=\sum_m(\tilde{a}^m_i)^2=(\tilde{a}^1_i)^2+1-(a^1_i)^2=1+(f^2-1)(a^1_i)^2.$$
We also have
$$a_i^1=e_i\cdot E_1=\frac{e_i\cdot X}{r},$$ where $X$ is the conformal Killing vector. Let $2\rho=r^2=X \cdot X$, then we get $$\rho_i=fe_i\cdot X=rf a^1_i.$$  Hence we obtain 
$$|e_i|_E^2=1+\frac{f^2-1}{2\rho f^2}\rho_i^2.$$ Thus \eqref{9.2} becomes
\begin{eqnarray}\label{9.3}
\int_{\Sigma}2\rho d\sigma\leq \int_{\Sigma}d\sigma+\int_{\Sigma} \frac{f^2-1}{2(n-1)\rho f^2}|\nabla\rho|^2d\sigma.\end{eqnarray}
In what follows, we will try to find another integral equality by Darboux equation. Using the orthonormal frame, by \eqref{2.6} we have
$$\sigma_2 (h)=\frac{(n-1)(n-2)}{2}+\sum_{i<j}\bar{R}_{ijij}.$$
 where the last term  is the curvature of the ambient space defined by \eqref{2.7}. Recall that $\varphi=X\cdot \nu$ is the support function of the convex surface, then we have
$$\varphi^2=2\rho-\frac{|\nabla \rho|^2}{f^2},\text{ and } (\nu^1)^2=(\nu\cdot E_1)^2=\frac{\varphi^2}{r^2}=1-\frac{|\nabla \rho|^2}{2\rho f^2}.$$ By \eqref{2.7}, we have
\begin{eqnarray}\label{9.4}
&&\frac{1}{n-2}\sigma_2 (h)\\
&=&\frac{n-1}{2}+ff_{\rho}+\frac{n-3}{2}\frac{f^2-1}{2\rho}-(1-\frac{|\nabla \rho|^2}{2\rho f^2})(ff_{\rho}+\frac{1-f^2}{2\rho})\nonumber\\
&=&\frac{n-1}{2}+\frac{n-1}{2}\frac{f^2-1}{2\rho}+\frac{f_{\rho}}{2\rho f}|\nabla \rho|^2+\frac{1-f^2}{4\rho^2f^2}|\nabla \rho|^2\nonumber,
\end{eqnarray}
 where we have used $$\frac{f'}{r}=f_{\rho}, \text{ and } 2\rho=r^2.$$ Denote $$w_{ij}=-\rho_{i,j}+\frac{f_{\rho}}{f}\rho_i\rho_j+f^2\delta_{ij}.$$  By \eqref{2.9}, we have $$h_{ij}=\frac{w_{ij}}{f\varphi}.$$
Then, by \eqref{9.4}, we have
\begin{eqnarray}\label{9.5}
\frac{1}{n-2}\frac{\sigma_2(w)}{f^2}&=&\varphi^2\left(\frac{n-1}{2}+\frac{1}{n-2}\sum_{i<j}\bar{R}_{ijij}\right)\\
&=&(n-1)\rho+\frac{n-1}{2}(f^2-1)+\frac{f_{\rho}}{f}|\nabla \rho|^2+\frac{1-f^2}{2\rho f^2}|\nabla \rho|^2\nonumber\\
&&-\left(\frac{n-1}{2}+\frac{1}{n-2}\sum_{i<j}\bar{R}_{ijij}\right)\frac{|\nabla \rho|^2}{f^2}.\nonumber
\end{eqnarray}
Let's calculate the left hand side.
Obviously, we have
\begin{eqnarray}
\sigma_2 (w)&=&\frac{1}{2}\sigma_2^{ij}(-\rho_{i,j}+\frac{f_{\rho}}{f}\rho_i\rho_j+f^2\delta_{ij})\nonumber\\
&=&\frac{1}{2}\sigma_2^{ij}(-\rho_{i,j})+\frac{f_{\rho}}{2f}\sigma_2^{ij}\rho_i\rho_j+\frac{(n-2)f^2}{2}(-\Delta \rho+\frac{f_{\rho}}{f}|\nabla \rho|^2+(n-1)f^2),\nonumber
\end{eqnarray}
hence we get
\begin{eqnarray}\label{9.6}
&&\int_{\Sigma}\frac{\sigma_2 (w)}{f^2}\\
&=&\int_{\Sigma}\left(\frac{\sigma_2^{ij} }{2f^2}\right)_{,j}\rho_i+\int_{\Sigma}\frac{f_{\rho}}{2f^3}\sigma_2^{ij}\rho_i\rho_j+\frac{n-2}{2}\int_{\Sigma}\frac{f_{\rho}}{f}|\nabla \rho|^2+\frac{(n-1)(n-2)}{2}\int_{\Sigma} f^2\nonumber\\
&=&\int_{\Sigma}\frac{1}{2f^2}(\sigma_2^{ij})_{,j}\rho_i-\int_{\Sigma}\frac{f_{\rho}}{2f^3}\sigma_2^{ij}\rho_i\rho_j+\frac{n-2}{2}\int_{\Sigma}\frac{f_{\rho}}{f}|\nabla \rho|^2+\frac{(n-1)(n-2)}{2}\int_{\Sigma} f^2\nonumber.
\end{eqnarray}
Let's calculate the first term of the above equality. We can rotate our frame to diagonalize the matrix $w_{ij}$,  then we have
\begin{eqnarray}\label{9.7}
(\sigma_2^{ij})_{,j}\rho_i=\sigma_2^{ij,pq}w_{pq,j}=\sum_{j\neq i}\rho_i(w_{jji}-w_{ijj}).
\end{eqnarray}
It is obvious that
\begin{eqnarray}\label{9.8}
w_{bja}-w_{baj}=\rho_{baj}-\rho_{bja}+\frac{f_{\rho}}{f}(\rho_{ba}\rho_j-\rho_{bj}\rho_a)+2ff_{\rho}(\rho_a\delta_{bj}-\rho_j\delta_{ab}).
\end{eqnarray}
By Ricci identity, we have
\begin{eqnarray}\label{9.9}
\rho_{baj}-\rho_{bja}=-\sum_c\rho_cR_{ajbc}.
\end{eqnarray}
By the definition of $w_{ij}$, we get
\begin{eqnarray}\label{9.10}
&&\frac{f_{\rho}}{f}(\rho_{ba}\rho_j-\rho_{bj}\rho_a)\\
&=&\frac{f_{\rho}}{f}[\rho_j(-w_{ba}+\frac{f_{\rho}}{f}\rho_b\rho_a+f^2\delta_{ba})-\rho_a(-w_{bj}+\frac{f_{\rho}}{f}\rho_b\rho_j+f^2\delta_{bj})]\nonumber.
\end{eqnarray}
Combing \eqref{9.8}-\eqref{9.10}, we obtain
\begin{eqnarray}
&&w_{bja}-w_{baj}\\
&=&\sum_c\rho_cR_{ajbc}+\frac{f_{\rho}}{f}(\rho_aw_{bj}-\rho_jw_{ba})+ff_{\rho}(\rho_a\delta_{bj}-\rho_j\delta_{ab}).\nonumber
\end{eqnarray}
Hence combing \eqref{9.7} and Einstein condition, we have
\begin{eqnarray}
(\sigma_2^{ij})_{,j}\rho_i&=&\sum_{i,c}\sum_{j\neq i}\rho_i\rho_cR_{ijjc}+\frac{f_{\rho}}{f}\sigma_2^{ij}\rho_i\rho_j+(n-2)ff_{\rho}|\nabla \rho|^2\nonumber\\
&=&(n-2)|\nabla \rho|^2+\frac{f_{\rho}}{f}\sigma_2^{ij}\rho_i\rho_j+(n-2)ff_{\rho}|\nabla \rho|^2.\nonumber
\end{eqnarray}
Thus, combing the above equality and \eqref{9.6}, we have
\begin{eqnarray}
\int_{\Sigma}\frac{\sigma_2 (w)}{f^2}&=&\frac{n-2}{2}\int_{\Sigma}\frac{|\nabla \rho|^2}{f^2}+(n-2)\int_{\Sigma}\frac{f_{\rho}}{f}|\nabla \rho|^2+\frac{(n-1)(n-2)}{2}\int_{\Sigma} f^2\nonumber.
\end{eqnarray}
Using \eqref{9.5} and the above equality, we have
\begin{eqnarray}\label{9.12}
\int_{\Sigma}2\rho d\sigma&=&\int_{\Sigma}d\sigma+\int_{\Sigma}\frac{|\nabla \rho|^2}{(n-1)f^2}d\sigma+\int_{\Sigma}\frac{f^2-1}{(n-1)\rho f^2}|\nabla \rho|^2d\sigma\\
&&+\int_{\Sigma} \left(1+\frac{2}{(n-1)(n-2)}\sum_{i<j}\bar{R}_{ijij}\right)\frac{|\nabla \rho|^2}{f^2}d\sigma.\nonumber
\end{eqnarray}
Combing \eqref{9.3} and \eqref{9.12}, we obtain
\begin{eqnarray}\label{9.13}
0&\geq &\int_{\Sigma}\frac{|\nabla \rho|^2}{(n-1)f^2}d\sigma+\int_{\Sigma}\frac{f^2-1}{2(n-1)\rho f^2}|\nabla \rho|^2d\sigma\\
&&+\int_{\Sigma} \left(1+\frac{2}{(n-1)(n-2)}\sum_{i<j}\bar{R}_{ijij}\right)\frac{|\nabla \rho|^2}{f^2}d\sigma.\nonumber
\end{eqnarray}
Now we are in the position to prove Theorem \ref{RC} .\\

\noindent {\bf Proof of Theorem \ref{RC}.}
Let's define some function depending on $\rho$
$$\phi(\rho)=2\rho+f^2-1.$$
Since function $\rho$ is defined on some convex body, $\rho$ should vary between its minimum value and maximum value. Denote the two values by $\rho_{\min},\rho_{\max},$ then $\rho_{\min}\leq \rho\leq \rho_{\max}$.
It is obvious that $$\phi_{\rho}=2(1+ff_{\rho}).$$
By \eqref{9.4}, we have
\begin{eqnarray}\label{9.14}
\frac{1}{n-2}\sigma_2 (h)=\frac{n-1}{2}\frac{\phi(\rho)}{2\rho}+\frac{\phi_{\rho}(\rho)}{4\rho f^2}|\nabla \rho|^2-\frac{\phi(\rho)}{4\rho^2f^2}|\nabla \rho|^2.
\end{eqnarray}
By the assumption of  convexity, the right hand side of the above equality is always positive. At the minimum point of $\rho$, $\nabla \rho=0,$ we have $\phi(\rho_{\min})>0$.
We claim $\phi$ is always positive between $\rho_{\min}$ and $\rho_{\max}$. If it is not true, let  $\rho_0$ be the first zero of $\phi$ from $\rho_{\min},$ then, at $\rho_0$,
$$\phi_{\rho}\leq 0.$$  By \eqref{9.14}, $\sigma_2 (h)\leq 0$ at $\rho=\rho_0$, since $ \phi_{\rho}\leq 0 $ and $\phi= 0,$ which contradicts to $\sigma_2 (h)>0$. Thus we always have $\phi>0$.

Note that \eqref{9.13} can be rewritten as
\begin{eqnarray}\label{9.15}
0\geq \int_{\Sigma}(\frac{\phi(\rho)}{2(n-1)\rho}+\frac{2\sigma_2(h)}{(n-1)(n-2)})\frac{|\nabla \rho|^2}{f^2}.
\end{eqnarray}
By the positivity of $\phi(\rho)$ and $\sigma_2(h)$, we obtain  $\nabla \rho=0.$ Hence $\rho$ is a constant.  Using \eqref{9.12} we have $2\rho=1$ which implies the embedded convex body is a unit geodesic sphere. We complete our proof. \\

In three dimension warped product space, obviously, we can have some statement about constant Gauss curvature.
\begin{coro}
In three dimensional warped product space, the only possible embedded convex body of the constant scalar curvature is the slice sphere provided that the embedded surface  satisfies condition \eqref{Cond}.
\end{coro}
\begin{proof}
Since constant scalar curvature implies the metric on $\mathbb{S}^2$ is the standard metric, an application of Theorem \ref{RC} immediately yields the corollary.
\end{proof}
If the warped function $f=1$, namely  Euclidean space, the above corollary is a classical rigidity theorem of two dimensional sphere.

\section{A Shi-Tam type inequality}
In this section, we try to generalize quasi-local mass in asymptotic Euclidean space, namely the case
$$f(r)=\sqrt{1-\frac{m}{r}}.$$ Assume the $\Omega$ be a domain bounded by the surface $\Sigma$ and $\Omega$ have nonnegative scalar curvature. Denote the mean curvature of $\Sigma$ in $\Omega$ by $H_1$. Furthermore, $\Sigma$ can be isometrically embedded in asymptotic Euclidean space as some surface $\Sigma_0$ containing the black hole. Now we consider the geodesic flow
\begin{eqnarray}
\left\{\begin{matrix}\label{10.1}
\dfrac{d \Phi(t,\cdot)}{dt}\ \ =\ \ &\nu&, \text{ if } t\in (0,+\infty)\\
\Phi(0,\cdot)\ \  = &\Sigma_0&
\end{matrix}\right.,
\end{eqnarray}
where $\Phi(t,\cdot)$ is the position vector and $\nu$ is the outward normal vector of the surface $\Phi(t,\cdot)$. Hence, we have some foliation of the space outside the $\Sigma_0$. The ambient space metric can be rewritten by
$$ds_0^2=dt^2+g_t,$$ where $g_t$ is the metric on $\Sigma_t$. We try to find  some conformal metric
$$ds^2=u^2dt^2+g_t $$ with the same scalar curvature as $ds_0^2$.  Denote  the mean curvature of the surface $\Sigma_t$ in $ds^2_0$ and $ds^2$ by $H_0$ and $H$.  Then we have the equation satisfied by $u$ \cite{ST1},
\begin{eqnarray}
\left\{\begin{matrix}\label{u}
H_0\dfrac{\p u}{\p t}\ \ =\  \  &u^2\Delta_t u&+\dfrac{u-u^3}{2}R^t, \text{ if }t\in (0,+\infty)\\
u(0)\ \ =&H_0/H_1&
\end{matrix}\right.,
\end{eqnarray}
where $R^t$ is the scalar curvature of surface $\Sigma_t$.

The following explicit formula for sectional curvatures is useful in this section.

\begin{lemm}
For any two vectors $Z_1,Z_2$ in $(\mathbb{R}^3,ds^2),$  we have
\begin{eqnarray}\label{sec}
\bar{R}(Z_1,Z_2,Z_1,Z_2)
=\frac{m}{2r^3}[|Z_1\times Z_2|^2-3(Z_1\times Z_2\cdot E_1)^2].
\end{eqnarray}
\end{lemm}

\begin{proof}
 Let's calculate the sectional curvature in detail.
Assume that $$Z_1=\sum_ia^iE_i, \text{ and } Z_2=\sum_jb^jE_j,$$ where $E_i$ is defined in \eqref{2.2} and $a^i,b^i$ are component of $Z_1, Z_2$. Hence by \eqref{curv}, we have
\begin{eqnarray}
\bar{R}(Z_1,Z_2,Z_1,Z_2)&=&\sum_{i,j,k,l}a_ia_kb_jb_l\bar{R}_{ijkl}=\sum_{i,j}((a^i)^2(b^j)^2-a^ia^jb^ib^j)\bar{R}_{ijij}\nonumber\\
&=&\sum_{i<j}(a^ib^j-a^jb^i)^2\bar{R}_{ijij}.\nonumber
\end{eqnarray}
It is obvious that
$$\bar{R}_{1212}=\bar{R}_{1313}=\frac{ff'}{r}=\frac{m}{2r^3},\text{ and } \bar{R}_{2323}=\frac{f^2-1}{r^2}=-\frac{m}{r^3}.$$
Hence we have
\begin{eqnarray}
\bar{R}(Z_1,Z_2,Z_1,Z_2)&=&\frac{m}{2r^3}\sum_{i<j}(a^ib^j-a^jb^i)^2+(a^2b^3-a^3b^3)^2(\bar{R}_{2323}-\frac{m}{2r^3})\nonumber\\
&=&\frac{m}{2r^3}[|Z_1\times Z_2|^2-3(Z_1\times Z_2\cdot E_1)^2]\nonumber.
\end{eqnarray}

\end{proof}

Using the above explicit formula, we can conclude that the geodesic flow preserves convexity.
\begin{lemm}\label{cp}
If the metric on the strictly convex $\Sigma_0$ is sufficiently close to a canonical sphere metric, then for any $t>0,$ every level
surface $\Sigma_t$ is always strictly convex.
\end{lemm}
\begin{proof}
Suppose $X$ is the conformal Killing vector, and $\varphi=X\cdot \nu$ is the support function.
Hence, we have
\begin{eqnarray}\label{r_t}
r_t=\frac{d r}{d t}=\frac{f\varphi}{r}.
\end{eqnarray}
Since  $\Phi(t,\cdot)$ is a geodesic flow, we have
\begin{eqnarray}\label{10.6}
D_{\frac{\p}{\p t}}\nu=0, \text{ and } \varphi_t=D_{\frac{\p}{\p t}}X\cdot \nu=f.
\end{eqnarray}
Combing \eqref{10.6} and \eqref{r_t}, we have $$(r^2-\varphi^2)_t=0.$$ Solving the above equation, we have
$$\varphi=\sqrt{r^2-C},$$ where $C$ is a positive constant only depending on $\Sigma_0$.  On the other hand,
$$\varphi^2=2\rho-\frac{|\nabla \rho|^2}{f^2}=2\rho-\frac{|\nabla \rho|^2}{1-m/r}$$ where $\rho=r^2/2$, which implies
 $C=\frac{|\nabla \rho|^2}{1-m/r}$. Note that, if $\Sigma_0$ is exactly the canonical sphere, we have $C=0.$ By continuity, we can
require $3C<r^2$ at $t=0$, since $\Sigma_0$ is sufficiently close to the canonical sphere.
Obviously, we have
$$|\nu\times E_1|^2=1-(\nu\cdot E_1)^2=1-(\frac{\varphi}{r})^2=\frac{C}{r^2}.$$ Thus  we have for any unit tangent  direction  $\mu$ on $\Sigma_t$,
\begin{eqnarray}
(\nu\times \mu\cdot E_1)^2=(\mu\cdot \nu\times E_1)^2\leq \frac{C}{r^2}.\nonumber
\end{eqnarray}
Taking $Z_1=\nu,Z_2=\mu$ in \eqref{sec}, we have
 \begin{eqnarray}\label{10.7}
 \bar{R}(\nu,\mu,\nu,\mu)&=&\frac{m}{2r^3}[|\nu\times\mu |^2-3(\nu\times \mu\cdot E_1)^2]\\
 &\geq& \frac{m}{2r^3}(1-3\frac{C}{r^2})>0,\nonumber
 \end{eqnarray}
 where we have used $\nu $ and $\mu$ are unit and perpendicular to each other.

By Riccati equation, the principal curvature $\lambda$ of $\Sigma_t$ satisfies
\begin{eqnarray}\label{Ricatti}
\frac{d}{dt}\lambda=-\lambda^2+\bar{R}(\nu,\mu,\nu,\mu),
\end{eqnarray}
where $\mu$ is the corresponding unit principal direction. By \eqref{10.7}, we have the convexity preservation.
\end{proof}

Now, let's consider the following quantity,
\begin{defi}
$$Q_t=\int_{\Sigma_t}(H_0-H)fd\sigma_t+\frac{m}{2}.$$
 \end{defi}
We will see that, if the metric on $\Sigma_0$ is sufficiently close to the canonical sphere metric, we have some monotonicity for $Q_t$.
\begin{lemm}
The quantity $Q_t$ is monotonically  decreasing along the geodesic flow.
\end{lemm}
\begin{proof}
The proof is modified from \cite{ST2}. Taking the trace of Riccati equation, we have
$$\frac{dH_0}{dt}=-|h|^2+\bar{Ric}(\nu,\nu),\text{ and } \frac{d \sigma_t}{dt}=H_0d\sigma_t,$$
where $h$ is the second fundamental form of $\Sigma_t.$

We can calculate the $Q_t$ using the parabolic equation \eqref{u} and \eqref{r_t},
\begin{eqnarray}
&&\frac{d}{dt}\int_{\Sigma_t}H_0(1-u^{-1})fd\sigma_t\nonumber\\
&=&\int_{\Sigma_t}(1-u^{-1})(H_0^2-|h|^2+\bar{Ric}(\nu,\nu))f d\sigma_t+\int_{\Sigma_t}(\Delta_t u+\frac{u^{-1}-u}{2}(2K))fd\sigma_t\nonumber\\
&&+\int_{\Sigma_t}H_0(1-u^{-1})\frac{ff'\varphi}{r}d\sigma_t\nonumber\\
&=&-\int_{\Sigma_t}\frac{u^{-1}}{2}(u-1)^2(2K)fd\sigma_t+\int_{\Sigma_t}[(u^{-1}-1)\bar{Ric}(\nu,\nu)f+u\Delta_t f]d\sigma_t\nonumber\\
&&+\int_{\Sigma_t}H_0(1-u^{-1})\frac{ff'\varphi}{r}d\sigma_t\nonumber,\end{eqnarray}
where $K$ is the Gauss curvature of $\Sigma_t$.

Recall the calculation in section $2$. By Gauss equation, we have
$$K=\bar{R}_{1212}=-\bar{Ric}(\nu,\nu)=\frac{ff'}{r}-\frac{\varphi^2}{r^4}(rff'+1-f^2).$$
By static equation, we have
$$0=\Delta f=f_{tt}+H_0f_t+\Delta_t f.$$ By \eqref{r_t} and \eqref{10.6}, we have
\begin{eqnarray}
f_t=\frac{ff'\varphi}{r}, \text{ and } f_{tt}=\frac{f'f^2}{r}+\frac{\varphi}{r^2}(f^2f''+f(f')^2)-\frac{f'f^2\varphi^2}{r^3}\nonumber.
\end{eqnarray}
Using the explicit formula for $f$, we have
$$-\bar{Ric}(\nu,\nu)=\frac{m}{2r^3}-\frac{3m\varphi^2}{2r^5},$$ and
$$\Delta_t f=f(-\frac{m}{2r^3}+\frac{3m\varphi^2}{2r^5})-H_0\varphi \frac{m}{2r^3}.$$
Thus we have
$$\int_{\Sigma_t}f(-\frac{m}{2r^3}+\frac{3m\varphi^2}{2r^5})d\sigma_t=\int_{\Sigma_t}H_0\varphi\frac{m}{2r^3}d\sigma_t.$$
Combing  the above integral equalities, we obtain
\begin{eqnarray}
&&\int_{\Sigma_t}[(1-u)\bar{Ric}(\nu,\nu)f+u\Delta_tf+H_0\varphi(1-u^{-1})\frac{ff'}{r}]d\sigma_t\nonumber\\
&=&\int_{\Sigma_t}f(-\frac{m}{2r^3}+\frac{3m\varphi^2}{2r^5})d\sigma_t+\int_{\Sigma_t}H_0\varphi\frac{m}{2r^3}(1-u-u^{-1})d\sigma_t\nonumber\\
&=&\int_{\Sigma_t}H_0\varphi\frac{m}{2r^3}d\sigma_t+\int_{\Sigma_t}H_0\varphi\frac{m}{2r^3}(1-u-u^{-1})d\sigma_t\nonumber\\
&=&-\int_{\Sigma_t}u^{-1}(1-u)^2H_0\varphi\frac{m}{2r^3}d\sigma_t\nonumber.
\end{eqnarray}
Thus we obtain
\begin{eqnarray}
&&\frac{d}{dt}\int_{\Sigma_t}H_0(1-u^{-1})fd\sigma_t\nonumber\\
&=&-\int_{\Sigma_t}\frac{u^{-1}}{2}(u-1)^2(2K-2\bar{Ric}(\nu,\nu))fd\sigma_t+\int_{\Sigma_t}[(1-u)\bar{Ric}(\nu,\nu)f+u\Delta_t f]d\sigma_t\nonumber\\
&&+\int_{\Sigma_t}H_0(1-u^{-1})\frac{ff'\varphi}{r}d\sigma_t\nonumber\\
&=&-\int_{\Sigma_t}\frac{u^{-1}}{2}(u-1)^2(2K-2\bar{Ric}(\nu,\nu))fd\sigma_t-\int_{\Sigma_t}u^{-1}(1-u)^2H_0\varphi\frac{m}{2r^3}d\sigma_t\nonumber.
\end{eqnarray}
Lemma \ref{cp} says that the geodesic flow can preserve the convexity, therefore the above formula is non positive. Hence we have the monotonicity.
\end{proof}

In what follows, we will discuss the asymptotic behavior of the principal curvature. By  \eqref{r_t} and $\varphi=\sqrt{r^2-C},$ we have
$$r_t=\sqrt{1-\frac{m}{r}}\sqrt{1-\frac{C}{r^2}}=\sqrt{1-\frac{m}{r}-\frac{C}{r^2}+\frac{mC}{r^3}}.$$
The right hand side is  bounded from below. It is easy to check that $r(t,\cdot)$ is increasing to infinity as $t$ approaches to infinity. Hence, for sufficiently  large $T,$ if $t>T$, we have
$$ r_t=1-\frac{m}{2r}+O(\frac{1}{r^2}),\text{ and } t-T=r-r(T)+\frac{m}{2}\log \frac{2r-m}{2r(T)-m}+O(\frac{1}{r}).$$
Using \eqref{10.7}, the increasing property of $r(t,\cdot)$ and \eqref{10.7}, we have,
 for $t\geq T$,
$$\frac{m}{4t^3}\leq R(\nu,\mu,\nu,\mu)\leq \frac{m}{t^3}.$$
Hence, we have
\begin{eqnarray}\label{10.9}
-\lambda^2+\frac{m}{4t^3}\leq \lambda_t\leq -\lambda^2+\frac{m}{t^3}.
\end{eqnarray}
Let $\phi=t\lambda-1$, then the right hand side of the above inequality can be rewritten as
$$(t\phi)_t\leq -\phi^2+\frac{m}{t}\leq \frac{m}{t}.$$ Integrating the above differential inequality, we have, for some $T,$ if $t>T,$
$$\lambda\leq \frac{1}{t}+\frac{m\log\frac{t}{T}}{t^2}+\frac{T\phi(T)}{t^2}\leq \frac{1}{t}+\frac{2m\log t}{t^2}.$$
By the left hand side of \eqref{10.9} and the above inequality, we have
$$\left(\frac{1}{\lambda}\right)_t=-\frac{\lambda_t}{\lambda^2}\leq 1-\frac{m}{4t^3\lambda^2}\leq 1-\frac{m}{5t}.$$
Similarly, we have, for some sufficiently large $T,$ if $t\geq T,$  $$\lambda\geq \frac{1}{t}+\frac{m\log t}{6t^2}.$$ So we obtain, for $t\geq T$,
$$ \lambda = \frac{1}{t}+O(\frac{\log t}{t^2}).$$
By \eqref{sec}, the sectional curvature of the ambient space along the tangential vectors is $O(\frac{1}{r^3})$. Hence by Gauss equations  we have the scalar curvature of the surface $\Sigma_t$,
$$R^t=\frac{2}{t^2}+O(\frac{\log t}{t^3}).$$

Next, we discuss the metric tensor on $\Sigma_t$. Let $\nu^t$ be the unit normal of  $\Sigma_t$. It is obvious that
$$D\nu^t=d\nu^t+O(\frac{1}{t})=d\frac{\nu^t}{|\nu^t|_E}+O(\frac{1}{t}).$$ Hence we have $$D\nu^t\cdot D\nu^t=d\frac{\nu^t}{|\nu^t|_E}\cdot_E d\frac{\nu^t}{|\nu^t|_E}+O(\frac{1}{t})=g_{\mathbb{S}^2}+O(\frac{1}{t}).$$ Thus we have
$$g|_{\Sigma_t}=t^2D\nu^t\cdot D\nu^t+O(t\log t)=t^2g_{\mathbb{S}^2}+O(t\log t).$$ Then the area element is
$$d\sigma_t=t^2d\sigma+O(t\log t),$$ where $d\sigma$ is the area element of $\mathbb{S}^2$.
Using some similar argument in \cite{ST1}, we have the following lemma,
\begin{lemm}\label{asym}
There exists unique solution to the initial value problem \eqref{u} in $[0,+\infty)$  and the asymptotic behavior of the solution  $u$ is
$$u(t)=1+\frac{m_0}{t}+O(\frac{\log t}{t^2}).$$
\end{lemm}
At last, let's calculate the ADM mass of the metric $ds^2$. Let $\{z^1,\cdots,z^n\}$ be the standard coordinate for $\mathbb{R}^n$.  The ADM mass of some metric $g$ is defined by,
$$ADM(g)=\lim_{r\rightarrow +\infty}\int_{\mathbb{S}^2}(\frac{\p g_{ij}}{\p z^i}-\frac{\p g_{ii}}{\p z^j})rz^jdV.$$
It is well known that
$$ADM(ds_0^2)=m \omega,$$ where $\omega$ is the area of standard sphere. Let $$b_{ij}=(u^2-1)\frac{\p t}{\p z^i}\frac{\p t}{\p z^j}.$$ A direct calculation shows
\begin{eqnarray}\label{10.10}
\\
\frac{\p b_{ij}}{\p z^i}-\frac{\p b_{ii}}{\p z^j}=2u\frac{\p t}{\p z^i}(\frac{\p u}{\p z^i}\frac{\p t}{\p z^j}-\frac{\p u}{\p z^j}\frac{\p t}{\p z^i})+(u^2-1)(\frac{\p^2 t}{\p (z^i)^2}\frac{\p t}{\p z^j}-\frac{\p^2 t}{\p z^i\p z^j}\frac{\p t}{\p z^i}).\nonumber
\end{eqnarray}
Since the flow is geodesic, it is obvious that
$$\frac{\p t}{\p z^i}=\frac{\p}{\p t}\cdot \frac{\p }{\p z^i},$$ and it is  bounded.
Using the asymptotic expression of $u$ in Lemma \ref{asym}, we have $$\frac{\p u}{\p z^i}=-\frac{m_0}{t^2}\frac{\p t}{\p z^i}+O(\frac{1}{t^2}).$$ Hence the first term in the right hand side of \eqref{10.10} is $O(\dfrac{\log t}{t^3})$ which will not appear in the ADM(b). It suffices to calculate the second term for $i\neq j$.

By the Levi-Civita property, we have
$$\frac{\p^2 t}{\p z^i\p z^j}=\frac{\p}{\p z^i}(\frac{\p}{\p t}\cdot\frac{\p}{\p z^j})=D_{\frac{\p}{\p z^i}}\frac{\p}{\p t}\cdot \frac{\p}{\p z^j}+\frac{\p}{\p t}\cdot D_{\frac{\p}{\p z^i}}\frac{\p}{\p z^j}.$$
Let $\lambda_1,\cdots,\lambda_{n-1}$ and $\{e_1,\cdots,e_{n-1}\}$ be the principal curvatures and corresponding unit principal directions, and we let $$\frac{\p}{\p z^i}=\sum_k a^k_ie_k+N_i\nu,$$ where $a_i^k, N_i$ are the component of $\dfrac{\p}{\p z_i}$ in the direction $e_k$ and $\nu$. Thus we have
$$D_{\frac{\p}{\p z^i}}\frac{\p}{\p t}\cdot \frac{\p}{\p z^j}=a_i^ka_j^lD_{e_k}\nu\cdot e_l=a_i^ka_j^l\lambda_k\delta_{kl}=\sum_ka_i^ka_j^k\lambda_k.$$ On the other hand, it is obvious that
$$g_{ij}=\frac{\p}{\p z^i}\cdot\frac{\p}{\p z^j}=\sum_{k}a^k_ia^k_j+N_iN_j=\sum_{k}a^k_ia^k_j+\frac{\p t}{\p z^i}\frac{\p t}{\p z^j}.$$ Using \eqref{8.6} and \eqref{8.7}, we get
\begin{eqnarray}
D_{\frac{\p}{\p z^i}}\frac{\p}{\p t}\cdot \frac{\p}{\p z^j}&=&\frac{1}{t}(\delta_{ij}+\psi z^iz^j-\frac{\p t}{\p z^i}\frac{\p t}{\p z^j})+O(\frac{\log t}{t^2})\\
&=& \frac{1}{t}(\delta_{ij}-\frac{\p t}{\p z^i}\frac{\p t}{\p z^j})+O(\frac{\log t}{t^2})\nonumber.
\end{eqnarray}
Since $\psi=O(1/t^3)$, we have $$D_{\frac{\p}{\p z^i}}\frac{\p}{\p z^j}=O(\frac{1}{t^2}).$$ Hence we have
\begin{eqnarray}
\frac{\p b_{ij}}{\p z^i}-\frac{\p b_{ii}}{\p z^j}&=&\frac{2m_0}{t}[\frac{1}{t}(1-(\frac{\p t}{\p z^i})^2)\frac{\p t}{\p z^j}+(\frac{\p t}{\p z^i})^2\frac{\p t}{\p z^i}]+O(\frac{\log t}{t^2})\nonumber\\
&=&\frac{2m_0}{t^2}\frac{\p t}{\p z^j}+O(\frac{\log t}{t^2})\nonumber.
\end{eqnarray}
Thus we have
$$ADM(b)=2m_0\lim_{t\rightarrow +\infty}\int_{\mathbb{S}^2}\sum_j\frac{z^j}{t}\frac{dt}{dr}\frac{\p r}{\p z^j}dV=2m_0\omega.$$
By the choice of $u$ in \eqref{u}, the scalar curvature of $ds^2$ is same as $ds_{0}^2$, hence it is nonnegative.
An application of  positive mass theorem \cite{ST1} for nonnegative scalar curvature space with Lipschitz metric immediately yields
$$ADM(ds^2)=ADM(ds_0^2)+ADM(b)=(2m_0+m)\omega\geq 0.$$ Now, we can prove the Theorem \ref{ST}.\\

\noindent {\bf Proof of Theorem \ref{ST}: }
By the monotonicity, it suffices to check that $2Q_t$ converges to $ADM(ds^2)$. It is easy to see that
$$\lim_{t\rightarrow +\infty}2Q_t=2\lim_{t\rightarrow +\infty}\int_{\mathbb{S}^2}H_0(1-u^{-1})\sqrt{1-\frac{m}{t}}d\sigma_t+m=2m_0+m\geq 0.$$
Thus we complete our proof.

\begin{rema} We believe that these type of inequality also holds in asymptotic hyperbolic manifold.
\end{rema}

\section{An example}
In the previous section, we have proved some inequality \eqref{LM} for some convex surfaces. Maybe an interesting problem is that under what restriction, we can drop the constant $m$? In this section, we will calculate this quantity using different sphere constructed in section $6$ in the same AdS space. It seem that the constant $m$ may be dropped, if the two surfaces live in a same space.

We use the notation in section $8$. Assume $\{z^1, z^2, z^3\}$ be the standard coordinate in $\mathbb{R}^3$.
We will write down the second order approximation of the convex perturbation surface defined in section $8$. In AdS space case, we have 
$$f(r)=\sqrt{1-\frac{m}{r}+\kappa r^2}.$$ Let $\vec{r}$ be the radius $1$ sphere and we have
$$\vec{y}=(\epsilon+\epsilon^2 h^1(\sin u)+\epsilon^3 h(\sin u))\frac{\p}{\p z^3},$$ then, by \eqref{W}, we have
\begin{eqnarray}
h^1(t)&=&\left.-\frac{1}{2\epsilon^2}\int^t_0W(1,s,\epsilon)ds\right|_{\epsilon=0}\\
&=&-\int^t_0ds\int^1_0(1-\tau)\psi(1)d\tau\nonumber\\
&=&-\frac{\psi(1)}{2}t\nonumber.
\end{eqnarray}
Using some similar argument in section $6$, we can find the function $h$ such that
$$d(\vec{r}+\vec{y})\cdot d(\vec{r}+\vec{y})=d\vec{r}\cdot d\vec{r}.$$
Denote
$$\alpha=-\frac{\psi(1)}{2}=\frac{f^2(1)-1}{2f^2(1)}=\frac{\kappa-m}{2(1-m+\kappa)}.$$ Then we have,
\begin{eqnarray}\label{11.2}
\rho(\epsilon)&=&\frac{1}{2}|\vec{r}+\vec{y}|^2_E=\frac{1}{2}(1+2\epsilon\sin u +\epsilon^2(1+2\alpha\sin^2u))+O(\epsilon^3).
\end{eqnarray}
By \eqref{2.9}, the second fundamental form is
\begin{eqnarray}\label{11.3}
h_{ij}=\frac{\rho_{i,j}-\frac{f_\rho}{f}\rho_i\rho_j-f^2g_{ij}}{f\varphi},
\end{eqnarray}
where $\varphi$ is the support function defined by $$\varphi=\sqrt{2\rho-\frac{|\nabla\rho|^2}{f^2}}.$$
Differentiation of \eqref{11.2} with respect to $u$, $v$ yields
\begin{eqnarray}\label{11.4}
\rho_u&=&\epsilon\cos u+\epsilon^2\alpha\sin 2u+O(\epsilon^3), \ \  \rho_v=O(\epsilon^3),\\
 \rho_{u,u}
 &=&-\epsilon\sin u+2\alpha\epsilon^2\cos 2u+O(\epsilon^3),\ \
 \rho_{u,v}=\rho_{uv}+\frac{\sin u}{\cos u}\rho_v=O(\epsilon^3), \nonumber\\
 \rho_{v,v}&=&\rho_{vv}-\cos u\sin u\rho_u=-\cos^2u(\epsilon\sin u +2\alpha\epsilon^2\sin^2u)+O(\epsilon^3)\nonumber.
\end{eqnarray}
Combining \eqref{11.3} and \eqref{11.4}, we have
\begin{eqnarray}
f\varphi h_{uu}&=&-\epsilon\sin u+2\epsilon^2\alpha\cos 2u-\frac{f'}{f}\epsilon^2\cos^2 u-f^2+O(\epsilon^3)\\
f\varphi h_{vv}&=&-\cos u\sin u(\epsilon\cos u+\epsilon^2\alpha\sin 2u)-f^2\cos^2 u+O(\epsilon^3)\nonumber.
\end{eqnarray}
Recall that the metric on unit sphere is $$ds^2=du^2+\cos^2 u dv^2.$$
The Taylor expansion of the support function and its inverse are respectively
\begin{eqnarray}
\varphi^2&=&1+2\epsilon\sin u+\epsilon^2(1+2\alpha\sin^2 u)-\epsilon^2\frac{\cos^2 u}{f^2}+O(\epsilon^3),\nonumber\\
\frac{1}{\varphi}&=&1-\epsilon\sin u+\epsilon^2(-\frac{1}{2}-\alpha\sin^2 u+\frac{\cos^2 u}{2f^2}+\frac{3\sin^2 u}{2})+O(\epsilon^3)\nonumber.
\end{eqnarray}
Define the mean curvature by  $H=-g^{ij}h_{ij}$, then, for the surface $\vec{r}+\vec{y}$, we have 
\begin{eqnarray}
\frac{H}{2}f&=&f^2-\epsilon(f^2-1)\sin u-\epsilon^2[f^2(\frac{1}{2}+\alpha\sin^2 u-\frac{\cos^2 u}{2f^2}-\frac{3}{2}\sin^2 u)\nonumber\\
&&+\alpha(\cos 2u-\sin^2 u)+\sin^2 u-\frac{f'}{2f}\cos^2 u]+O(\epsilon^3).\nonumber
\end{eqnarray}
We also have the following expansion,
\begin{eqnarray}
\frac{1}{\sqrt{2\rho}}&=& 1-\epsilon\sin u+\frac{\epsilon^2}{2}(3\sin^2u-1-2\alpha\sin^2u)+O(\epsilon^3)\nonumber\\
f^2&=&1-m+\kappa+(m+2\kappa)\epsilon\sin u\nonumber\\
&&+\frac{\epsilon^2}{2}(m+2\kappa+(4\kappa\alpha+2m\alpha-3m)\sin^2u)+O(\epsilon^3)\nonumber\\
\sqrt{1-m+\kappa}f&=&1-m+\kappa+\frac{m+2\kappa}{2}\epsilon\sin u\nonumber\\
&&+\frac{\epsilon^2}{4}(m+2\kappa+(4\kappa\alpha+2m\alpha-3m)\sin^2 u)-\frac{(m+2\kappa)^2}{8(1-m+\kappa)}\epsilon^2\sin^2 u. \nonumber
\end{eqnarray}
Hence we get
\begin{eqnarray}
&&\frac{H}{2}f-\sqrt{1-m+\kappa}f\nonumber\\
&=&\frac{3m}{2}\epsilon\sin u+\epsilon^2[m(\alpha-1)\sin^2 u+\kappa(2\alpha+1)\sin^2 u-f^2(\alpha\sin^2 u-\frac{\cos^2 u}{2f^2})\nonumber\\
&&-\sin^2 u+\frac{f'\cos^2 u}{2f}-(\frac{m}{4}+\frac{\kappa}{2})(1+2\alpha \sin^2 u)+(\frac{3m}{4}+\frac{(m/2+\kappa)^2}{2f^2})\sin^2 u]\nonumber\\
&&+\epsilon^2(\frac{m}{2}+\kappa-\frac{f^2}{2}-\alpha)(1-3\sin^2 u)+O(\epsilon^3)\nonumber.
\end{eqnarray}
In the previous formula $f,f',\alpha$ are all constant, since they take value in $r=1$. It is easy to calculate that
$$\int_{\mathbb{S}^2}d\sigma=4\pi, \int_{\mathbb{S}^2}\sin^2 ud\sigma=\frac{4\pi}{3},\int_{\mathbb{S}^2}\cos^2 ud\sigma=\frac{8\pi}{3},\int_{\mathbb{S}^2}(1-3\sin^2 u)d\sigma=0.$$
Using the above equalities, we have the following expansion
\begin{eqnarray}
&&\frac{6}{4\pi}\int_{\mathbb{S}^2}[\frac{H}{2}f-\sqrt{1-m+\kappa}f]d\sigma\nonumber\\
&=&\epsilon^2[2m(\alpha-1)+2\kappa(2\alpha+1)-f^2(2\alpha-\frac{4}{2f^2})-2+\frac{4f'}{2f}\nonumber\\&&-(\frac{m}{4}+\frac{\kappa}{2})(6+4\alpha)+2(\frac{3m}{4}+\frac{(m/2+\kappa)^2}{2f^2})]+O(\epsilon^3)\nonumber\\
&=&m\alpha-2m+2\alpha\kappa-\kappa-2\alpha f^2(1)+\frac{m+2\kappa}{f^2(1)}+\frac{(m+2\kappa)^2}{4f^2(1)}+O(\epsilon^3)\nonumber\\
&=&\frac{3m}{4}\frac{m+2\kappa}{1-m+\kappa}\epsilon^2+O(\epsilon^3)\nonumber.
\end{eqnarray}
Since $2\sqrt{1-m+\kappa}$ is the mean curvature of slice sphere, the quantity calculated above is nothing but the mass defined by \eqref{LM} without extra $m/2$. Thus we have, if  $\epsilon$ is sufficient small,
$$\int_{\mathbb{S}^2}[H-2\sqrt{1-m+\kappa}]fd\sigma>0.$$

\begin{ack} The authors would like to thank Professor Pengfei Guan for drawing them to study the Weyl problem in warped product space. The paper can not be finished without his continuous encouragement. Part of the content in this paper also comes from his contribution.  They
would like to thank Professor Chao Xia, Professor Yiyan Xu and Doctor Siyuan Lu for their stimulate discussion and valuable comments. The work was done while the two authors are visiting McGill University. They want to thank China Scholarship Council for its financial support. They also wish to thank McGill University for their hospitality.
\end{ack}


\begin{thebibliography}{99}

\bibitem{A}  A. D.  Alexandrov,{\em On a class of closed surfaces}, Mat. Sb., 1938, 4: 69-77.
\bibitem{B1} W. Blaschke, {\em Vorlesungen Uber Differentialgeometrie}, Vol. 1, Julis Springer, Berlin, 1924.
\bibitem{B2} S. Brendle, {\em Constant mean curvature surfaces in warped product space}, Publ. Math. Inst. Hautes etudes Sci.(117) 2013, 247-269.

\bibitem{BY} D. J. Brown and J. W. York, {\em Quasilocal energy and conserved changes derived from the gravitational action,} Phys. Rev. D(3) 47(1993), no.4, 1407-1419.


\bibitem{CX} J. Chang and L. Xiao, {\em The Weyl problem with nonnegative Gauss curvature in hyperbolic space, }  Canadian Journal of Mathematics http://dx.doi.org/10.4153/CJM-2013-046-7.
\bibitem{CV1}  E. Cohn-Vossen, {\em Zwei Satze uber die Starrheit der Eiflachen, Nach. Gesell-schaft Wiss}, Gottingen, Math. Phys.
KL, 1927: 125.134.
\bibitem{CV2} E. Cohn-Vossen, {\em Unstarre geochlossene Flachen}, Math. Ann., 1930, 102: 10.29.


\bibitem{D} M. Dajczer, {\em Submanifolds and Isometric immersions}, Mathematical Lecture Series 13, 1990.
\bibitem{DR} M. Dajczer, and L. L. Rodriguez {\em Infinitesimal rigidity of Euclidean submanifolds}, Ann. Inst. Four. (1990), p. 939-949.

\bibitem{Ei} L.P. Eisenhart, {\em Reimannian geometry}, Princeton University Press, 1997.
\bibitem{GL1} P. Guan and S. Lu, {\em Curvature estimats for immersed hypersurfaces in Reimannian manifolds}, preprint.
\bibitem{GL2} P. Guan and Y. Y.  Li, {\em On Weyl problem with nonnegavie Gauss curvature}, J. Diff. Geom.,
39(1994), 331-342.
\bibitem{GM} M. Giaquinta and L. Martinazzi, {\em An Introduction
to the Regularity Theory for Elliptic Systems, Harmonic Maps
and Minimal Graphs}, Edizioni Della Normale, 2012 Scuola Normale Superiore Pisa, Seconda edizione.

\bibitem{GS} P. Guan and X. Shen, {\em A rigidity theorem for hypersurfaces in higher dimensional space forms,}
Contemporary Mathematics, AMS. V.644, 2015. pp. 61-65.

\bibitem{GWZ} P. Guan, Z. Wang, and X. Zhang {\em A proof of the Alexanderov's uniqueness theorem for convex surface in $\mathbb{R}^3$}, Ann. Inst. H. Poincare Anal. Non Lineaire,  (2016) 329-336.

\bibitem{H} R.S. Hamilton, {\em The Ricci flow on surface}, Mathematics and general relativity (Santa Cruz, CA, 1986), 237-262, Contemp. Math. 71, Amer. Math. Soc., Providence, RI, 1988.

\bibitem{HH} Q. Han and J.-X. Hong,  {\em Isometric embedding of Riemanian manifolds in Euclidean spaces}, mathematical survey and monographs, Vol. 130. American mathematical society.

\bibitem{HNY} Q. Han, N. Nadirashvili and Y. Yuan, {\em Linearity of homogeneous order-one solutions to elliptic equations in dimension three}.  Comm. Pure and App. Math.  Vol LVI(2003), 0425-0433.


\bibitem{HZ} J. Hong and C. Zuily, {\em Isometric embedding of the 2-sphere with nonnegative curvature in $\mathbb{R}^3$},
Math. Z., 219(1995), 323-334.

\bibitem{HIII} L. H\"ormander {\em The analysis of linear partial differential operators III}, Grundlehren der mathematischen Wissenschaften 274, Springer-Verlag, 1994.

\bibitem{I}
J. A. Iaia, {\em Isometric embeddings of surfaces with nonnegative curvature in $\mathbb{R}^3$}, Duke Math. J.,
Vol 67, (1992), 423-459.

\bibitem{K} S. Kobayashi, {\em Transformation groups in differential geometry,} Springer, 1995.
\bibitem{L} H. Lewy, {\em On the existence of a closed convex sucface realizing a given Riemannian metric},
Proceedings of the National Academy of Sciences, U.S.A., Volume 24, No. 2, (1938), 104-106.
\bibitem{LW} C.-Y. Lin and Y.-K. Wang, {\em On isometric embeddings into anti-de sitter spacetimes}, Int. Math. Res. Notices, No. 16 (2015), 7130-7161.

\bibitem{LY1} C.-C. M. Liu and S.T. Yau, {\em Positivity of quasilocal mass}, Phys. Rev. Lett. 90(2003) No. 23,231102, 4p.
\bibitem{LY2} C.-C. M. Liu and S.T. Yau, {\em Positivity of quasilocal mass II}, J. Amer. Math. Soc. 19(2006) No. 1, 181-204.

\bibitem{Lu} S. Lu, {\em Curvature estimates for Weyl problem in Reimannian manifolds}, preprint.

\bibitem{M} S.B. Myers, {\em Riemannian manifolds with positive mean curvature}, Duke Math. J. 8(1941) No.2, 401-404

\bibitem{N1} L. Nirenberg, {\em The Weyl and Minkowski problems in differential geometry in the large},  Comm. Pure  Appl. Math.
(1953), p.337-394.
\bibitem{N2} L. Nirenberg, {\em Nonlinear Problems}, Madison: University of Wisconsin Press, 1963, 177-193.

\bibitem{Ob} M. Obata. {\em Certain conditions for a Reimannian metric to constant scalar curvature}, J. Math. Soc. Japan. 14(1962), 333-340.
\bibitem{P1} A. V. Pogorelov, {\em Regularity of a convex surface with given Gaussian curvature}, (Russian)
Mat. Sbornik N.S. 31(73), (1952). 88-103.
\bibitem{P2} A.V. Pogorelov,  {\em Some results on surface theory in the large},  Adv.  Math. 1 1964,
fasc. 2, 191-264.

\bibitem{PM}  A.J. C. Pacheco and P. Miao, {\em Isometric embeddings of 2-spheres into Schwarzschild manifolds}, Manuscripta Math.(2015), doi:10.1007/s00229-015-0782-2.
\bibitem{R} A. Ros, {\em Compact hypersurfaces with constant scalar curvature and a congruence theorem. With an appendix by Nicholas J. Korevaar.} J. Diff. Geom. 27 (1988), no. 2, 215-223.

\bibitem{ST1} Y.  Shi and L-F. Tam, {\em Positive mass theorem and the boundary behaviors of compact manifolds with nonnegative scalar curvature.} J. Diff. Geom. 62 (2002), no. 1, 79-125.

\bibitem{ST2} Y.  Shi and L-F. Tam, {\em Rigidity of compact manifolds and positivity of quasi-local mass.} Class. Quantum Grav. 24 (2007) 2357-2366.

\bibitem{V} I. N. Vekua, {\em Generalized analytic functions}, English translation, Pergamon Press, 1962.
\bibitem{W} H. Weyl, {\em Uber die Bestimmung einer geschlossenen konvexen Flache durch ihr Linienelement},
Vierteljahrsschrift der naturforschenden Gesellschaft, Zurich, 61, (1916), 40-72.

\bibitem{WY1} M.-T. Wang and S.-T. Yau {\em A generalization of Liu-Yau's quamsilocal mass}, Comm. Anal.  Geom.,V.15,(2007), 249-282.


\bibitem{WY2} M.-T. Wang and S.-T. Yau {\em Isometric embedding into Minkowski space and new quasi-local mass}, Comm. Math. Phys. 288(3), 919-942(2009).

\bibitem{WY3} M.-T. Wang and S.-T. Yau {\em Quasilocal mass in general relativity}, Phys. Rev. Lett. (2009). doe:10.1103/PhysRevLett. 102.021101.

\end{thebibliography}
\end{document}